\documentclass[sn-mathphys]{sn-jnl}

\jyear{2023}
\theoremstyle{thmstyleone}%
\newtheorem{theorem}{Theorem}
\theoremstyle{thmstyletwo}%
\newtheorem{remark}{Remark}%

\theoremstyle{thmstylethree}%
\newtheorem{definition}{Definition}%
\newtheorem{lemma}{Lemma}

\usepackage[shortlabels]{enumitem}

\begin{document}

\title[Boundary Recovery of Anisotropic Electromagnetic Parameters]{Boundary Recovery of Anisotropic Electromagnetic Parameters for the Time Harmonic Maxwell's Equations}


\author[1]{\fnm{Sean} \sur{Holman}}\email{sean.holman@manchester.ac.uk}

\author[2]{\fnm{Vasiliki} \sur{Torega}}\email{vasiliki.torega@manchester.ac.uk}

\affil[1]{\orgdiv{Department of Mathematics}, \orgname{University of Manchester}, \orgaddress{\street{Oxford Road}, \city{Manchester}, \postcode{M13 9PL}, \state{Greater Manchester}, \country{UK}}}

\affil[2]{\orgdiv{Department of Mathematics}, \orgname{University of Manchester}, \orgaddress{\street{Oxford Road}, \city{Manchester}, \postcode{M13 9PL}, \state{Greater Manchester}, \country{UK}}}


\abstract{This work concerns inverse boundary value problems for the time-harmonic Maxwell's equations on differential $1-$forms. We formulate the boundary value problem on a $3-$dimensional compact and simply connected Riemannian manifold $M$ with boundary $\partial M$ endowed with a Riemannian metric $g$. Assuming that the electric permittivity $\varepsilon$ and magnetic permeability $\mu$ are real-valued anisotropic (i.e $(1,1)-$ tensors), we aim to determine certain metrics induced by these parameters, denoted by $\hat{\varepsilon}$ and $\hat{\mu}$ at $\partial M$. We show that the knowledge of the impedance and admittance maps determines the tangential entries of $\hat{\varepsilon}$ and $\hat{\mu}$ at $\partial M$ in their boundary normal coordinates, although the background volume form cannot be determined in such coordinates due to a non-uniqueness occuring from diffeomorphisms that fix the boundary. Then, we prove that in some cases, we can also recover the normal components of $\hat{\mu}$ up to a conformal multiple at $\partial M$ in boundary normal coordinates for $\hat{\varepsilon}$. Last, we build an inductive proof to show that if $\hat{\varepsilon}$ and $\hat{\mu}$ are determined at $\partial M$ in boundary normal coordinates for $\hat{\varepsilon}$, then the same follows for their normal derivatives of all orders at $\partial M$.}

\keywords{Maxwell's equations, Inverse Problems, Anisotropic, Electromagnetic Parameters, Boundary Normal Coordinates, Pseudodifferential Operators}


\maketitle

\section{Introduction}\label{sec1}
In this work, we consider a $3-$dimensional compact and simply connected Riemannian manifold $M$,  with smooth boundary $\partial M$, equipped with a Riemannian metric $g$. In the case of $M$ $\subset$ $\mathbb{R}^{3}$, $g$ could be the Euclidean metric. The coordinate invariant formulation of time harmonic Maxwell's equations, in the absence of currents or sources, in $M$ is provided by
\begin{equation}\label{Maxg}
\begin{split}
\ast_{g}\mathrm{d}H= -i \omega \varepsilon E, \quad \delta_{g} (\varepsilon E)= 0,\\
\ast_{g}\mathrm{d}E= i \omega \mu H, \quad \delta_{g} (\mu H)= 0,
\end{split}
\end{equation}
where $\ast_{g}$ is the Hodge star operator with respect to the metric $g$ and $\mathrm{d}$ stands for the exterior derivative. The operator $\delta_{g}$ denotes the divergence operator with respect to the  metric $g$.  The electric and magnetic fields $E$ and $H$, respectively, are differential $1-$forms. The electric permittivity $\varepsilon$ and magnetic permeability $\mu$ are assumed to be real-valued $(1,1)-$tensor fields which are symmetric and positive definite with respect to the metric $g$. The angular frequency is represented by $\omega$ and is a real number.

We consider either of the Dirichlet boundary conditions
\begin{equation}\label{pullbackBCs}
\iota^{\ast}H=F \quad \mbox{or} \quad  \iota^{\ast}E=G,
\end{equation}
where $\iota^{\ast}$ stands for the pullback at $\partial M$. Given that the Dirichlet problem composed of \eqref{Maxg} and either boundary condition in \eqref{pullbackBCs} is well-posed, the boundary mappings $\Lambda_{\varepsilon}$ and $\Lambda_{\mu}$, called the impedance and admittance maps respectively are well defined and are provided by
\begin{equation*}
\Lambda_{\varepsilon}(\iota^{\ast}H)=\iota^{\ast}E \quad \mbox{and} \quad \Lambda_{\mu}(\iota^{\ast}E)=\iota^{\ast}H.
\end{equation*}
Note that $\Lambda_{\mu}$ is the inverse of $\Lambda_{\varepsilon}$.
Assuming the knowledge of $\Lambda_{\varepsilon}$ and $\Lambda_{\mu}$ we aim to recover the metrics induced by $\varepsilon$ and $\mu$ on the cotangent bundle of $M$, denoted by $\varepsilon^{\sharp}$ and $\mu^{\sharp}$, as well as their normal derivatives at $\partial M$. 

\subsection{Literature Review}

The inverse problem of recovering a Riemannian metric using boundary data has been widely studied for the wave equation. The boundary measurements in this case are realised by the Dirichlet to Neumann ($DtN$) map. A natural obstruction to the uniqueness of the reconstruction, also mentioned as ``cloaking", is encountered in this type of inverse problems and is caused by diffeomorphisms that fix the boundary \cite{invisibility}, \cite{2danisotropic}.
The article \cite{lassas2014inverse} studies the inverse problem of reconstructing a Riemannian manifold through the $DtN$ map for the wave equation.  The Dirichlet boundary conditions represent the sources and they are located in an open subset of the boundary $S$. The Neumann boundary conditions which represent the waves produced by the sources in $S$, are being observed in $R$, which is an open subset of the boundary, with $\overline{S}\cap \overline{R}=\emptyset$. The method employed to solve this inverse problem uses arguments of the boundary control method. Applications of the boundary control method are given in \cite{belishev2007recent} which concerns inverse problems in acoustics, elasticity, electrodynamics, electrical impedance tomography and graph theory.

Well-posedness results for the time harmonic Maxwell's equations are given in \cite{somersalolinearized} in the case of isotropic electromagnetic parameters and in \cite{Keniganisotropic} in the case of anisotropic electromagnetic parameters that are conformal multiples. Both the isotropic and anisotropic cases are considered in \cite{costabel1991coercive} where a coercive billinear form for the decoupled system is constructed. The time domain is also taken into account in \cite{uniquenessanisotropic} and the well-posedness is proved for non-analytic anisotropic parameters that are conformal multiples of each other. In this paper, we prove well-posedness for the equations considered by augmenting Maxwell's system to an elliptic one. Further details are provided later in sections \ref{resultssec} and \ref{wellposednesssec}.

The linearized inverse boundary value problem of Maxwell's equations is studied in \cite{somersalolinearized}. The electric permittivity, conductivity and magnetic permeability are assumed to be functions that differ slightly from given constants and they are recovered through the admittance map.
 The case of isotropic (scalar) electromagnetic parameters is also considered in \cite{Mcdowall1997}. In this case the Maxwell's system can be decoupled to derive a second order differential equation in terms of the Laplace-Beltrami operator either with respect to the magnetic field or the electric field. Using the calculus of Pseudodifferential Operators \cite{shubin1987pseudodifferential}, \cite{treves1980introduction} the symbols of the admittance and impedance maps are determined in terms of a pseudodifferential operator of order $1$ that factorises the Laplace-Beltrami operator. It is shown that this implies the boundary recovery of the electric permittivity $\varepsilon$ and magnetic permeability $\mu$ and their first normal derivatives. This result is extended in \cite{Joshi} where the demonstration of an inductive proof establishes the boundary determination of the normal derivatives of all orders of the electromagnetic parameters. This method uses a special class of Pseudodifferential Operators that depends smoothly on the distance from the boundary. Similar reasoning is followed in \cite{Lionheart} for the inverse problem of the Laplacian on $k-$forms where the boundary data are represented by the $DtN$ operator.
 
 A different approach on the inverse problem of time harmonic Maxwell's equations for scalar electromagnetic parameters is demonstrated in \cite{Somersalo1994Calderon}. The principal symbols of the impedance and admittance maps are derived using a boundary operator called the Calderon projector for an augmented system equivalent to the Maxwell's system. 
 The isotropic inverse problem for the time harmonic Maxwell's equations is also studied in \cite{chiralMcDowal} where the medium is a subset of $\mathbb{R}^{3}$ and is assumed to be chiral. In this method the magnetic permeability is known and the recovery of the electric permittivity, conductivity and chirality constant throughout the medium of interest is proven. The augmented system used in \cite{Somersalo1994Calderon} to construct solutions for this system and use complex geometrical optics, is also used in the present work to establish well-posedness.  The inverse problem for the metrics induced by $\varepsilon$ and $\mu$ being $(0,2)-$tensors and conformal multiples of each other is studied in \cite{Keniganisotropic}. To the best of our knowledge, \cite{Keniganisotropic} is the only anisotropic case regarding the recovery of $\varepsilon$ and $\mu$ studied in the existing literature.
 
 The article \cite{Leeconductivities} concerns the anisotropic conductivity problem. The method employed is the same factorisation method used for the isotropic case of the time harmonic Maxwell's equations in \cite{Mcdowall1997}. In the case of the dimension $n$ satisfying $n\geq 3$ the conductivity problem can be reduced to the problem of recovering a Riemannian metric $g$ using the $DtN$ map for harmonic functions. 
The same factorisation method is also applied to the elasticity problem for transversely isotropic materials in  \cite{Nakamura1997layer}. The elasticity operator, which is a second order elliptic operator, is factorised in terms of a pseudodifferential operator of order $1$ making use of the theory of monic matrix polynomials \cite{Lancaster}. The aforementioned operator is used to express the $DtN$ map which yields to the recovery of the elastic parameters at the boundary. This is the approach that we follow in this paper in order to calculate the principal symbols of the boundary mappings.

\subsection{Summary of Current Work}
 
Now let us describe the contents of the present work. We start in section \ref{backgroundsec} by introducing the background material and notation including new metrics $\hat{\varepsilon}$ and $\hat{\mu}$ which are conformally related to $\varepsilon^{\sharp}$ and $\mu^\sharp$ but also eliminate the dependence of Maxwell's equations on $g$. Then follows section \ref{resultssec}, where we state our main results. In order to show that the impedance and admittance maps are well defined we start in section \ref{wellposednesssec} by illustrating the proof of well-posedness of the forward problem. To do so, we augment the Maxwell's system to an $8 \times 8$ elliptic system also used in \cite{Somersalo1994Calderon} and \cite{chiralMcDowal}. The basic step, following \cite{taylor1996partial}, is to prove a coercivity estimate for the elliptic operator of the system.

In section \ref{nonuniquenesssec}, we start by defining the impedance and admittance maps. 
First, we study two types of non-uniqueness that arise from diffeomorphisms that fix the boundary. We show that if $(\hat{\varepsilon},\hat{\mu})$,  $(\hat{\varepsilon}',\hat{\mu}')$ are two different sets of metrics connected by such a diffeomorphism $\Phi$ then the boundary data of the corresponding Dirichlet problems are the same.
Next, we examine the consequences of a diffeomorphism $\Phi$ that fixes the boundary to the recovery of the determinant of the metric $g$, denoted by $\lvert g\rvert$, in boundary normal coordinates for $\varepsilon^{\sharp}$ and $\mu^{\sharp}$. We prove that the Jacobian determinant $\lvert D\Phi \rvert $ of this diffeomorphism can be arbitrary in boundary coordinates and as a result, the determinant $\lvert g \rvert $ cannot be uniquely determined in either boundary normal coordinates for $\varepsilon^{\sharp}$ or $\mu^{\sharp}$. This affects the boundary recovery of $\varepsilon^{\sharp}$ and $\mu^{\sharp}$ as we show in section \ref{nonuniquenesssec}.

 In section \ref{systemdecouplingsec} we decouple the system and follow \cite{Nakamura1997layer} together with the theory of monic matrix polynomials \cite{Lancaster} to derive the principal symbols of the factorisation operators denoted by $B(x,D_{\tilde{x}})$ and $C(x,D_{\tilde{x}})$ representing the normal derivatives in coordinates. Next, we express the principal symbols of the impedance and admittance maps through the principal symbols of $B$ and $C$. Proceeding to section \ref{boundaryrecoverysec}, we show that the principal symbols of $\Lambda_{\varepsilon}$ and $\Lambda_{\mu}$ determine the tangential entries of the metrics $\hat{\varepsilon}$ and $\hat{\mu}$. Next, investigating the consequences of the boundary recovery on the electric and magnetic fields, we prove that in some cases we can recover the normal components of $\hat{\mu}$ in boundary normal coordinates for $\hat{\varepsilon}$. Under the same choice of coordinates, we show that when boundary recovery of the metrics $\hat{\varepsilon}$, $\hat{\mu}$ is possible, the boundary recovery of their normal derivatives of all orders is implied. This is done by adjusting to our case the inductive proof employed in \cite{Joshi} and \cite{Lionheart} that uses Riccati equations satisfied by $B(x, D_{\tilde{x}})$, $C(x,D_{\tilde{x}})$ together with a class of pseudodifferential operators that allows us to work exclusively with principal symbols at the boundary.
 
 \section{Background material}\label{backgroundsec}
 
 Let us introduce some prerequisites from tensor calculus that will be needed in the following sections. Further details on the topic are given in \cite{lee2013smooth}.
 We will denote by $\tau_{n}^{m}(M)$ the spaces of smooth $(n,m)-$ complex valued tensor fields on $M$. The Riemmanian $g$ induces a point-wise billinear form $\langle \cdot, \cdot \rangle_{g}$ on these spaces. Also, we introduce the space $\Omega^{k}(M)$ that consists of the smooth complex-valued differential $k-$forms. 
 For $a,b$ $\in$ $\Omega^{k}(M)$ we will write the point-wise inner product of $a$ with $b$ as $\langle a, \bar{b}\rangle_{g}$.  
 For $a,b$ $\in$ $\Omega^{1}(M)$ we express $\langle a, \bar{b}\rangle_{g}$ in a local coordinate chart as
 \begin{equation*}
\langle a, \bar{b}\rangle_{g}= g^{ab}a_{a}\bar{b}_{b},
 \end{equation*}
 which induces the norm $\langle a, \bar{a}\rangle_{g}^{1/2}$.
 Note carefully that, in our notation, for complex valued fields the bilinear form $\langle a,b \rangle_g$ is not an inner-product because we do not include the conjugation. Indeed, in local coordinates it is written as
 \begin{equation*}
 \langle a, b \rangle_{g}= g^{ab}a_{a}b_{b}.
 \end{equation*}
 Furthermore, we will write $\lvert a \rvert _{g}^{2} = \langle a, a \rangle_{g}$, which for complex $a$ does not give a norm.
 
 \begin{remark}\label{innerproductremark}
 The purpose of excluding the conjugation from $\langle a, b \rangle_{g}$ and $\lvert a \rvert _{g}^2$ is that in what follows we will make use of the fact that $\lvert a \rvert _{g}^2$ can vanish for complex $a$. When we reach this point, the reader is reminded with a remark (see Remark \ref{innervanish}).
 \end{remark}
 
 \noindent The inner product on $L^{2}(\Omega^{k}(M))$ is given by
 \begin{equation*}
 \langle \langle a,\bar{b} \rangle\rangle _{g}= \int\limits_{M} a \wedge \ast_{g} \bar{b},
 \end{equation*}
 where $a,b$ $\in$ $\Omega^{k}(M)$. In what follows, we will also use inner products defined as above but with $g$ replaced by metrics induced from the electromagnetic parameters $\varepsilon$ and $\mu$.
 Using the Levi-Civita connection on $(M,g)$, denoted by $\nabla$, we define the spaces $H^{1}(\Omega^{k}(M))$ as 
 \begin{equation*}
 H^{1}(\Omega^{k}(M))= \{a \in L^{2}(\Omega^{k}(M)); \quad \nabla a \in L^{2}(\tau_{k+1}^{0}(M))\},
 \end{equation*}
 equipped with the norm
 \begin{equation*}
 \|a \| _{H^{1}(\Omega^{k}(M))}^{2}= \| a \| _{L^{2}(\Omega^{k}(M))}^{2}+ \| \nabla a \|_{L^{2}(\tau_{k+1}^{0}(M))}^{2}.
 \end{equation*}
 Let $\iota: \partial M \mapsto M$ be the inclusion map at the boundary and define the pullback $\iota^{\ast}: \Omega^{k}(M) \mapsto \Omega^{k}(\partial M)$ on smooth differential forms. The pullback can be extended by continuity to $\iota^{\ast}: H^{1}(\Omega^{k}(M))\mapsto H^{1/2}(\Omega^{k}(\partial M))$.
 The definition of $H^{1/2}$ and in general fractional Sobolev spaces is given in section $4.1$ of \cite{taylor1996partial}, in terms of the Fourier transform.
 Moreover, we define the Sobolev space $H_{0}^{1}(\Omega^{k}(M))$ as
 \begin{equation*}
 H_{0}^{1}(\Omega^{k}( M))= \{ a \in H^{1}(\Omega^{k}(M)); \quad \iota^{\ast}a=0\}.
 \end{equation*}
  The $(1,1)-$tensors $\varepsilon$ and $\mu$ are assumed to be smooth and real valued. We interpret them as point-wise maps on the cotangent bundle $T^{\ast}(M)$ and assume that they are symmetric and positive definite with respect to $g$. For $a,b$ $\in$ $T_{x}^{\ast}(M)$ these properties are described by, respectively,
 \begin{equation*}
 \langle a, \varepsilon \bar{b} \rangle_{g} = \langle \varepsilon a, \bar{b} \rangle_{g}, \ \langle a, \mu \bar{b} \rangle_{g} = \langle \mu a, \bar{b} \rangle_{g},
 \end{equation*} 
 \begin{equation*}
 \langle a, \varepsilon \bar{a} \rangle_{g} >0, \quad \langle a, \mu \bar{a} \rangle_{g} >0 \quad \mbox{whenever $a \neq 0$.}
 \end{equation*}
 Due to the above, $\varepsilon$, $\mu$ and their inverses $\varepsilon^{-1}$, $\mu^{-1}$ induce metrics on the cotangent and tangent bundles $T^{\ast}(M)$ and $T(M)$, respectively. The metrics on $T^{\ast}(M)$ induced by $\varepsilon$ and $\mu$ are $(2,0)-$tensors and using the $\sharp$ (raising indices) operator are provided in a local coordinate by
 \begin{equation*}
 \varepsilon^{\sharp}=\varepsilon^{ij}=g^{ik}\varepsilon_{k}^{j}, \quad \mu^{\sharp}=\mu^{ij}=g^{ik}\mu_{k}^{j}.
 \end{equation*}
We will write the inverses of $\varepsilon^{\sharp}$ and $\mu^{\sharp}$ using the $\flat$ (lowering indices) operator as
 \begin{equation*}
 (\varepsilon^{-1})_{\flat}=(\varepsilon^{-1})_{ij}=g_{ik}(\varepsilon^{-1})_{j}^{k}, \quad (\mu^{-1})_{\flat}=(\mu^{-1})_{ij}=g_{ik}(\mu^{-1})_{j}^{k}.
 \end{equation*}
Note that we will often use the term metric both for Riemannian metrics, which are $(0,2)-$tensors and induce a bilinear form on $T(M)$, as well as for $(2,0)-$tensors and the corresponding bilinear forms on $T^{\ast}(M)$.

 \subsection{Notation and identities}
 Let us introduce the notation used in this paper. We are working in a set of coordinates adapted to the boundary $\partial M$ characterised by $x_{3}=0$ within the domain of the coordinate chart. These will be boundary normal coordinates either for $\varepsilon^{\sharp}$, $\mu^{\sharp}$ or the related metrics $\hat{\varepsilon}$, $\hat{\mu}$ introduced in \eqref{hatmetrics}. Since we are working on a $3-d$ manifold, the indices of tensor fields $i$ are assumed to take values from $1$ to $3$. In the case of an index restricted to $i=1,2$, we will add a tilde as in $\tilde{i}$ .
 When working with coordinate invariant expressions we will use the notation $\tilde{\xi}$ for the covector
 \begin{equation*}
 \xi_{\tilde{i}} \ dx^{\tilde{i}} = \sum_{\tilde{i}=1}^2  \xi_{\tilde{i}} \ dx^{\tilde{i}}.
 \end{equation*}
 If $\nu= \ dx^{3}$ is a covector normal to the boundary, we will denote by $\xi$ a covector that is written in components as
 \begin{equation*}
 \xi=\tilde{\xi}+\nu \xi_{3}.
 \end{equation*}
 The determinants of the Riemannian metrics $g$, $(\varepsilon^{-1})_{\flat}$ and $(\mu^{-1})_{\flat}$ in coordinates will be written as
 \begin{equation*}
 \lvert g \rvert :=\det(g), \quad \lvert \varepsilon^{-1}\rvert:=\det((\varepsilon^{-1})_{\flat}), \quad \lvert \mu^{-1}\rvert:=\det((\mu^{-1})_{\flat}).
 \end{equation*}
 The corresponding notation for the determinants of the metrics on $T^{\ast}(M)$ is
 \begin{equation*}
 \lvert g^{-1}\rvert:=\det(g^{-1}), \quad \lvert \varepsilon\rvert:=\det(\varepsilon^{\sharp}), \quad \lvert \mu\rvert:=\det(\mu^{\sharp}).
 \end{equation*}
Denoting by $\sigma^{ijk}$ the alternating tensor will use the following identities for the metric $g^{-1}$
 \begin{equation}\label{altdetinvid}
 \lvert g^{-1}\rvert \sigma^{pqr}=\sigma_{ijk}g^{pi}g^{qj}g^{rk},
 \end{equation}
 and for the matrix of cofactors of the Riemannian metric $g$ 
 \begin{equation}\label{identitycofact}
 \sigma^{aqj}\sigma^{dkb}g_{bj}= \lvert g \rvert (g^{ad}g^{qk}-g^{ak}g^{qd}).
 \end{equation}
 We define new Riemannian metrics $\hat{\varepsilon}^{-1}$ and $\hat{\mu}^{-1}$ by
  \begin{equation}\label{hatmetrics}
 \frac{\hat{\varepsilon}^{-1}}{\sqrt{\lvert\hat{\varepsilon}^{-1}\rvert}}= \frac{(\varepsilon^{-1})_{\flat}}{\sqrt{\lvert g \rvert}}, \quad \frac{\hat{\mu}^{-1}}{\sqrt{\lvert\hat{\mu}^{-1}\rvert}}= \frac{(\mu^{-1})_{\flat}}{\sqrt{\lvert g \rvert}},
 \end{equation}
 which allow us to write the time harmonic Maxwell's equations as given in \eqref{Maxg}
 in the equivalent form
 \begin{equation}\label{Maxhat}
 \begin{split}
 \ast_{\hat{\varepsilon}}\mathrm{d}H&= -i \omega E, \quad \delta_{\hat{\varepsilon}}  E= 0,\\
 \ast_{\hat{\mu}}\mathrm{d}E&= i \omega  H, \quad \delta_{\hat{\mu}} H= 0,\\
 \end{split}
 \end{equation}
 that is independent of $g$.  When working with the metrics \eqref{hatmetrics}, we will denote the associated impedance and admittance maps by $\Lambda_{\hat{\varepsilon}}$, $\Lambda_{\hat{\mu}}$, respectively. The approach we follow is to work with two different pairs of metrics $(\hat{\varepsilon},\hat{\mu})$, $(\hat{\varepsilon}',\hat{\mu}')$ and boundary mappings $\Lambda_{\hat{\varepsilon}}, \Lambda_{\hat{\mu}}$ and $\Lambda_{\hat{\varepsilon}'}$, $\Lambda_{\hat{\mu}'}$ to investigate the implications of the assumption $\Lambda_{\hat{\varepsilon}}=\Lambda_{\hat{\varepsilon}'}$, $\Lambda_{\hat{\mu}}=\Lambda_{\hat{\mu}'}$. 
 We will use the phrases ``boundary normal coordinates for $\hat{\varepsilon}$/$\hat{\varepsilon}'$" and ``boundary normal coordinates for $\hat{\mu}$/$\hat{\mu}'$" to express that we work separately in boundary normal coordinates for either the metrics $\hat{\varepsilon}$ and $\hat{\varepsilon}'$ or $\hat{\mu}$ and $\hat{\mu}'$. Assume that the representations of $\hat{\varepsilon}$ and $\hat{\varepsilon}'$ in boundary normal coordinates for $\hat{\varepsilon}$/$\hat{\varepsilon}'$ at $x_{3}=0$ are given by, respectively
 \begin{equation}\label{epsilonepsilon'tilde}
 \left(\begin{matrix}
 \tilde{\varepsilon}&0\\
 0^{T}& 1
 \end{matrix}\right), \quad  	\left(\begin{matrix}
 \tilde{\varepsilon}'&0\\
 0^{T}& 1
 \end{matrix}\right).
 \end{equation}
 Assume further that $\hat{\mu}$ and $\hat{\mu}'$ in boundary normal coordinates for $\hat{\mu}$/$\hat{\mu}'$ at $x_{3}=0$ are written as
 \begin{equation}\label{mumu'tilde}
 \left(\begin{matrix}
 \tilde{\mu}&0\\
 0^{T}& 1
 \end{matrix}\right), \quad 	\left(\begin{matrix}
 \tilde{\mu}'&0\\
 0^{T}& 1
 \end{matrix}\right).	
 \end{equation}
Note that $\tilde{\varepsilon}$, $ \tilde{\varepsilon}'$, $\tilde{\mu}$ and $\tilde{\mu}'$ define coordinate invariant metrics on $\partial M$. These are the tangential parts of  $\hat{\varepsilon}$, $ \hat{\varepsilon}'$, $\hat{\mu}$ and $\hat{\mu}'$ respectively.

As mentioned in Section \ref{sec1}, we are going to use the calculus of Pseudodifferential Operators $(\Psi DOs)$ to determine the principal symbols of the boundary mappings $\Lambda_{\hat{\varepsilon}}$ and $\Lambda_{\hat{\mu}}$. Using the Fourier transform and its inverse, $P$ is a pseudodifferential operator of order $m$ ($P \in \Psi DO^{(m)}$) in local coordinates if
\begin{equation}
P(x,D_{x})u(x)= \frac{1}{(2\pi)^{3}} \int\limits_{\mathbb{R}^{3}} e^{i x \cdot \xi} p(x,\xi)\hat{u}(\xi) \ d\xi,
\end{equation}
where $D_{x}= - i \partial_{x}$, $\hat{u}(\xi)$ denotes the Fourier transform of $u(x)$ and $p(x,\xi)$ is called the symbol of $P(x,D_{x})$ and is of order $m$ with respect to the dual variable $\xi$. For details on the definition and calculus of $\Psi DO$s, which will be used, see \cite{shubin1987pseudodifferential,treves1980introduction}.

\section{Results}\label{resultssec}
In this section we present the main results included in this paper. The content of the remainder of the paper contains the proofs for each result.

\begin{theorem}[Well-posedness of the forward problem.]\label{wellposednessthm}
	For $\omega$ $\in$ $\mathbb{R}$, outside a discrete set, the solution of the Dirichlet problem 
	\begin{equation*}
	\begin{split}
	\ast_{\hat{\varepsilon}}\mathrm{d} H&= -i \omega E, \quad \delta_{\hat{\varepsilon}}E=0,\\
	\ast_{\hat{\mu}}\mathrm{d} E&= i \omega H, \quad \delta_{\hat{\mu}}H=0,
	\end{split}
	\end{equation*}
	\begin{equation*}
	\iota^{\ast}H=F, 
	\end{equation*}
	exists, is unique and depends continuously on the data $F$. More particularly, there exists $\hat{C}>0$ that depends on $\hat{\varepsilon}$, $\hat{\mu}$ and $\omega$, such that
	\begin{equation*}
	\|E\|_{H^{1}(\Omega^{1}(M))}^{2}+ \|H\|_{H^{1}(\Omega^{1}(M))}^{2} \leq \hat{C} \|F\|_{H^{1/2}(\Omega^{1}(\partial M))}^{2}.
	\end{equation*}
\end{theorem}

\begin{proof} 
See Section \ref{wellposednesssec}.
\end{proof}

\begin{theorem}[Non-uniqueness.]\label{nonuniquenessthm}
If $(\hat{\varepsilon},\hat{\mu})$ and $(\hat{\varepsilon}',\hat{\mu}')$ are related by a diffeomorphism fixing the boundary of $M$ then $\Lambda_{\hat{\varepsilon}}=\Lambda_{\hat{\varepsilon}'}$. Additionally, let $g_{\varepsilon}$, $g_{\varepsilon'}$, $g_{\mu}$ and $g_{\mu'}$ be the local coordinate representations of the metric $g$ in boundary normal coordinates for $\varepsilon^{\sharp}$/$\varepsilon'^{\sharp}$, $\mu^{\sharp}$/$\mu'^{\sharp}$. There exist pairs of metrics $(\varepsilon^{\sharp},\mu^{\sharp})$, $(\varepsilon'^{\sharp},\mu'^{\sharp})$ not equal to each other such that $\Lambda_\epsilon =\Lambda_{\epsilon'}$ and the ratios of the determinants in the different sets of boundary normal coordinates $\frac{\lvert g_{\varepsilon}\rvert}{\lvert g_{\varepsilon'}\rvert}$, $\frac{\lvert g_{\mu}\rvert}{\lvert g_{\mu'}\rvert}$ are arbitrarily chosen positive functions in a neighbourhood of the boundary.
\end{theorem}

\begin{proof} 
See Section \ref{sec:nonunique}.
\end{proof}

\begin{theorem}[Boundary recovery of tangential components of the metrics.]\label{metricsequalinbncsthm}
	Let $(\hat{\varepsilon},\hat{\mu})$, $(\hat{\varepsilon}',\hat{\mu}')$ be two different sets of metrics and $\Lambda_{\hat{\varepsilon}}$, $\Lambda_{\hat{\varepsilon}'}$ the associated impedance maps. Assuming that $\Lambda_{\hat{\varepsilon}}=\Lambda_{\hat{\varepsilon}'}$, then 
	\begin{equation*}
	\tilde{\varepsilon}=\tilde{\varepsilon}', \quad \mbox{at} \ x_{3}=0,
	\end{equation*}
	\begin{equation*}
	\tilde{\mu}=\tilde{\mu}', \quad \mbox{at} \ x_{3}=0.
	\end{equation*}
\end{theorem}

\begin{proof} 
See Section \ref{metricsequalinbncsec}.
\end{proof}

\begin{theorem}[Boundary recovery of the full metrics.]\label{normalmuthm}
	Let $(\hat{\varepsilon},\hat{\mu})$ $(\hat{\varepsilon}',\hat{\mu}')$ be two different sets of metrics and $\Lambda_{\hat{\varepsilon}}$, $\Lambda_{\hat{\varepsilon}'}$ the associated boundary mappings. Let us fix boundary normal coordinates for $\hat{\varepsilon}$/$\hat{\varepsilon}'$.
	If $\Lambda_{\hat{\varepsilon}}=\Lambda_{\hat{\varepsilon}'}$, then at each point in $\partial M$ least one of the following cases holds:
	\begin{enumerate}
		\item The metrics $\tilde{\mu}$ and $\tilde{\varepsilon}$ are multiples and $\hat{\mu}'^{3\tilde{j}}=c \hat{\mu}^{3\tilde{j}}$ for some $c \in \mathbb{R}$.
		\item $\displaystyle \hat{\varepsilon}=\hat{\varepsilon}'$ and $\hat{\mu}=\hat{\mu}'$.
	\end{enumerate}
\end{theorem}

\begin{proof} 
See Section \ref{normalmusec}.
\end{proof}

\begin{theorem}[Boundary recovery of normal derivatives of the metrics.]\label{normalderivativesthm}
	Let $(\hat{\varepsilon},\hat{\mu})$, $(\hat{\varepsilon}',\hat{\mu}')$ be two different sets of metrics and $\Lambda_{\hat{\varepsilon}}$, $\Lambda_{\hat{\varepsilon}'}$ the corresponding boundary mappings. Let us fix boundary normal coordinates for $\hat{\varepsilon}$/$\hat{\varepsilon}'$. If
	\begin{equation*}
	\Lambda_{\hat{\varepsilon}}=\Lambda_{\hat{\varepsilon}'},
	\end{equation*}
	\begin{equation*}
	\hat{\varepsilon}=\hat{\varepsilon}', \quad \hat{\mu}=\hat{\mu}' \quad \mbox{at} \ x_{3}=0,
	\end{equation*}
	then,
	\begin{equation*}
	\partial_{x_{3}}^{\kappa}\hat{\varepsilon}=	\partial_{x_{3}}^{\kappa}\hat{\varepsilon}', \quad 	\partial_{x_{3}}^{\kappa}\hat{\mu}=	\partial_{x_{3}}^{\kappa}\hat{\mu}' \quad \mbox{at} \ x_{3}=0,
	\end{equation*}
	for any $\kappa\geq 1$.
\end{theorem}

\begin{proof} 
See Section \ref{normalderivativessec}.
\end{proof}

\section{Sketch of the proof of Theorem \ref{wellposednessthm}}\label{wellposednesssec}
Our proof of Theorem \ref{wellposednessthm} relies on augmenting Maxwell's equations to a larger elliptic system and proving a coercivity estimate for this system. Since the methods are fairly standard and well-known, we only give a sketch of this proof.

To begin, we augment the system \eqref{Maxhat} to
\begin{equation}\label{Peq}
PU=	-i \left (
\begin{matrix}
i \omega & 0 & \delta_{\hat{\mu}} & 0 \\
0 & i \omega &  \ast_{\hat{\varepsilon}} \mathrm{d} &  \mathrm{d} \\
\mathrm{d} &-\ast_{\hat{\mu}} \mathrm{d} & i \omega & 0\\
0 & \delta_{\hat{\varepsilon}} & 0 & i \omega
\end{matrix}
\right )
\left (
\begin{matrix}
u_{E} \\ E \\ H \\ u_{H}
\end{matrix}
\right )
= 
\left (
\begin{matrix}
0 \\ 0 \\ 0 \\ 0
\end{matrix}
\right ),
\end{equation}
where $u_{E}$ and $u_{H}$ are auxiliary $0-$forms. Note that if $u_{E}=u_{H}=0$, then $E$ and $H$ satisfying \eqref{Peq} also satisfy \eqref{Maxhat}. The purpose of this step is that $P$ is an elliptic operator, i.e its principal symbol is invertible for $\xi$ $\in$ $\mathbb{R}^{3}\backslash\{0\}$, in contrast to the operator corresponding to the system \eqref{Maxhat}. The field $U$ is defined in $\Omega(M)$, where 
\begin{equation*}
\Omega(M)= \Omega^{0}(M) \oplus \Omega^{1}(M) \oplus \Omega^{1}(M) \oplus \Omega^{0}(M)
\end{equation*}
and the boundary condition is written in terms of $U$ as
\begin{equation}\label{Bbcs}
BU= \left( \begin{matrix}
\iota^{\ast}&0&0&0\\
0&0&0&0\\
0&0&\iota^{\ast}&0\\
0&0&0&\iota^{\ast}
\end{matrix}\right)\left (
\begin{matrix}
u_{E} \\ E \\ H \\ u_{H}
\end{matrix}
\right )= \left(\begin{matrix}0\\0\\F\\0  \end{matrix}\right).
\end{equation}
The spaces $L^{2}(\Omega(M))$ and $H^{1}(\Omega(M))$ are defined using the inner products
and norms given from the product structure
\begin{equation}\label{Omegainner}
\begin{split}
(U,V)&=		\int_{M} \left( \begin{matrix}
u_{E}\\
E\\
H\\
u_{H}
\end{matrix} \right)\wedge \left( \begin{matrix}
\ast_{\hat{\mu}}&0&0&0\\
0&\ast_{\hat{\varepsilon}}&0&0\\
0&0&\ast_{\hat{\mu}}&0\\
0&0&0&\ast_{\hat{\varepsilon}}
\end{matrix} \right)\left( \begin{matrix}
\overline{u_{E}}\\
\overline{E}\\
\overline{H}\\
\overline{u_{H}}
\end{matrix} \right),
\end{split}
\end{equation}
\begin{equation*}
\left\| \left( \begin{matrix}
u_{E}\\
E\\
H\\
u_{H}
\end{matrix}
\right)\right \|^{2}= \|u_{E}\|^{2} + \|E\|^{2}+\|H\|^{2}+ \|u_{H}\|^{2},
\end{equation*}
where the norms are either $L^{2}$ or $H^{1}$.
Transforming the boundary value problem consisting of \eqref{Peq} and \eqref{Bbcs} to the equivalent one with non-homogeneous differential equation and homogeneous Dirichlet condition we define the domain of $P$ to be
\begin{equation*}
D(P)= \{U \in H^{1}(\Omega(M)); \quad BU=0\}.
\end{equation*}
The main step in the proof is to prove the following coercivity estimate which holds for any $U$ $\in$ $D(P)$ 
\begin{equation}\label{Pineq}
\|U\|_{H^{1}(\Omega(M))}^{2} \leq C \|PU\|_{L^{2}(\Omega(M))}^{2} + C \|U\|_{L^{2}(\Omega(M))}^{2},
\end{equation}
for some $C>0$, depending on $\hat{\varepsilon}$, $\hat{\mu}$ and $\omega$.  In \cite[Theorem $11.1$]{taylor1996partial} is shown that the ellipticity of $P$ implies that \eqref{Pineq} holds in the interior of $M$. Thus, it remains to show that the estimate \eqref{Pineq} holds in a coordinate chart adapted to the boundary. Using the method in \cite[Proposition $11.2$]{taylor1996partial}, it suffices to show that $P$ satisfies the hypothesis of regularity upon freezing the coefficients.   The only difference is that we consider the eigenvalues and eigenvectors of $K_{1}(\tilde{\xi})+K_{0}$ instead of $K_{1}(\tilde{\xi})$ as denoted in the proof of \cite[Proposition $11.2$]{taylor1996partial} to accomodate the fact that the eigenvectors of $K_{1}(\tilde{\xi})$ in our case do not satisfy \cite[Lemma $11.7$]{taylor1996partial}.
Next, using the symmetry of $P$ with respect to the  the $L^{2}-$product defined in \eqref{Omegainner} and using the coercivity estimate we show that $P$ is densely defined self-adjoint with compact resolvent. Hence, the set of eigenvalues of $P$ is discrete. Finally, to obtain the result for the original time harmonic Maxwell's equations we must show that, for values of $\omega$ outside the spectrum, if $u_{E} \vert_{\partial M}= u_{H}\vert_{\partial M} = 0$ then $u_E = u_H = 0$. This is done by showing that $u_E$ and $u_H$ satisfy eigenvalue problems which can be decoupled from the augmented system.
\section{Boundary mappings and non-uniqueness result}\label{nonuniquenesssec}
The fact that the boundary value problem with Dirichet condition on $H$ is well-posed outside a discrete set of $\omega$ implies the corresponding result for the boundary value problem with Dirichlet condition
\begin{equation*}
\iota^{\ast}E=G.
\end{equation*} 
Therefore, we can define the impedance map $\Lambda_{\hat{\varepsilon}}: \Omega^{1}(\partial M) \mapsto \Omega^{1}(\partial M)$ realised by
\begin{equation*}
\Lambda_{\hat{\varepsilon}}(\iota^{\ast}H)=\iota^{\ast}E
\end{equation*}
and its inverse $\Lambda_{\hat{\mu}}:\Omega^{1}(\partial M) \mapsto \Omega^{1}(\partial M)$ given by
\begin{equation*}
\Lambda_{\hat{\mu}}(\iota^{\ast}E)=\iota^{\ast}H
\end{equation*}
which is called the admittance map.
Next we study the non-uniqueness implied by diffeomorphisms $\Phi$ that fix the boundary and thus prove Theorem \ref{nonuniquenessthm}.

\subsection{Proof of Theorem \ref{nonuniquenessthm}}\label{sec:nonunique}

Let us consider the sets of metrics $(\hat{\varepsilon},\hat{\mu})$ and $(\hat{\varepsilon}',\hat{\mu}')$ connected by a diffeomorphism $\Phi:$ $M\mapsto M$ with $\Phi \vert_{\partial M}= \mathbb{I}$ as
\begin{equation}\label{pullbacks}
\hat{\varepsilon}= \Phi^{\ast} \hat{\varepsilon}', \quad \hat{\mu}=\Phi^{\ast} \hat{\mu}'.
\end{equation}
Because of the coordinate invariant form of \eqref{Maxhat}, the action of $\Phi$ on the magnetic field $H$ is
\begin{equation*}
H= (D\Phi)^{T}H'
\end{equation*}
and since $\Phi \vert_{\partial M}= \mathbb{I}$ we get $\iota^{\ast}H=\iota^{\ast}H'$. The same is also true for the electrial field: $\iota^{\ast}E=\iota^{\ast}E'$.
Therefore the impedance maps corresponding to the $\hat{\varepsilon}$, $\hat{\varepsilon}'$ metrics satisfy
\begin{equation*}
\Lambda_{\hat{\varepsilon}}(\iota^{\ast}H)=\Lambda_{\hat{\varepsilon}'}(\iota^{\ast}H).
\end{equation*}
Hence, the metrics $(\hat{\varepsilon},\hat{\mu})$ and $(\hat{\varepsilon}',\hat{\mu}')$ lead to the same boundary data, which implies the same for the metrics $(\varepsilon^{\sharp},\mu^{\sharp},g)$ and $(\varepsilon'^{\sharp},\mu'^{\sharp},g)$ defined through \eqref{hatmetrics}.

To complete the proof of Theorem \ref{nonuniquenessthm} we need the following Lemma.

\begin{lemma}\label{philemma}
Suppose that $\{x^j\}_{j=1}^3$ are a set of local boundary coordinates on a domain $U \subset M$ and take any positive function $h \in C^\infty(M)$ such that $\mathrm{supp}(h-1) \subset U$. Then there exists a diffeomorphism $\Phi:M \mapsto M$ such that
	\begin{equation*}
	\Phi{\vert}_{\partial M}=\mathbb{I}
	\end{equation*}
and in the local coordinates the Jacobian determinant satisfies $\lvert D\Phi \rvert = h$ for all $x^3$ sufficiently small.
\end{lemma}
\begin{proof}
We will assume without loss of generality that
\[
\mathrm{supp}(h-1) \subset \{(x^1,x^2,x^3) \in [-\epsilon,\epsilon]\times[-\epsilon,\epsilon]\times [0,\epsilon]\} \subset U.
\]
for some $\epsilon >0$. Let $\phi \in C^\infty_c(\mathbb{R})$ be an even function satisfying $0\leq \phi \leq 1$, $\mathrm{supp}(\phi) \subset (-1,1)$, $\phi(s) = 1$ for $s$ in a neighourhood of zero and $\int_{-\infty}^\infty \phi(s) \ \mathrm{d} s =1$. Also set
\[
a = 4\ \mathrm{sup}\{ h+1\} + 3, \quad b = \frac{2\epsilon}{a} \left ( \ \mathrm{sup}\{ h+1\} + 1\right ), \quad c = \frac{2 \epsilon}{a} \mathrm{sup}\{ h+1\},
\]
\[
d(x^1,x^2) = \frac{1}{c} \int_0^\infty (h(x^1,x^2,s)-1) \phi \left ( \frac{as}{\epsilon} \right ) \ \mathrm{d} s. 
\]
Then we define $\Phi$ within $U$ by
	\begin{equation*}
	\Phi(x^{1},x^{2},x^{3})=\left (x^{1},x^{2}, f(x^{1},x^{2},x^{3})\right),
	\end{equation*}
	with
\begin{equation}\label{defa3}
f(x^{1},x^{2},x^{3})=\int_{0}^{x^{3}} (h(x^{1},x^{2},s)-1)\phi\left (\frac{as}{\epsilon}\right ) - d(x^1,x^2) \phi\left ( \frac{s-b}{c} \right ) +1\  \mathrm{d}s.
\end{equation}
With this construction, we see first that $f(x^{1},x^{2},x^{3})$ satisfies $f(x^{1},x^{2},0)=0$ which implies that
$\Phi\vert_{\partial M}= \mathbb{I}$. Also, noting that the support of the first and second terms in the integrand on the right side of \eqref{defa3} is contained in $\{s \in [0,\epsilon]\}$ we can conclude
\begin{equation}\label{a3}
\frac{\partial f}{\partial x^3}(x^{1},x^{2},x^{3}) =  (h(x^{1},x^{2},x^3)-1)\phi\left (\frac{ax^3}{\epsilon}\right ) - d(x^1,x^2) \phi\left ( \frac{x^3-b}{c} \right ) +1 > 0
\end{equation}
everywhere and $f(x^1,x^2,x^3) = x^3$ for $(x^1,x^2,x^3) \notin [-\epsilon,\epsilon]\times[-\epsilon,\epsilon]\times [0,\epsilon]$. Thus, $\Phi$ can be extended as the identity outside $U$ to give a diffeomorphism on $M$ with Jacobian matrix given within $U$ by
	\begin{equation}\label{DPhi}
	D\Phi= \left(\begin{matrix}
	1&0& 0\\
	0&1&0\\
	\frac{\partial f}{\partial x^1} & \frac{\partial f}{\partial x^2} &\frac{\partial f}{\partial x^3}
	\end{matrix}
	\right).
	\end{equation}
From \eqref{a3}, we see that $\frac{\partial f}{\partial x^3} = h$ when $x^3$ is sufficiently small, and so the Jacobian determinant is $\lvert D\Phi \rvert = h$ for all such $x^3$. This completes the proof.
\end{proof}

Now let us proceed to the proof of the second part of Theorem \ref{nonuniquenessthm} using the diffeomorphism $\Phi$ constructed in Lemma \ref{philemma} in \eqref{pullbacks}. Given a set of coordinates on the boundary, let $\Psi$, $\Psi'$ be the corresponding coordinate transformations from $M$ to boundary normal coordinates for $\hat{\varepsilon}$ and $\hat{\varepsilon}'$, respectively. Let us denote by $\hat{g}$ and $\hat{g}'$ the local representations of the metric $g$ in boundary normal coordinates for $\hat{\varepsilon}/\hat{\varepsilon}'$. Then, referring to Figure \ref{fig1} and using the fact that $\Psi=\Psi' \circ \Phi$, we deduce that 
\begin{equation}\label{gg'hat}
\lvert	\hat{g}\rvert=\lvert D\Phi \rvert^{-2}\lvert\hat{g}'\rvert.
\end{equation}
	Using the result of Lemma \ref{philemma}, that $\lvert D\Phi \rvert$ can be an arbitrary positive function in a neighbourhood of the boundary, we conclude that the expression for the volume form of $g$ in boundary normal coordinates for $\hat{\varepsilon}$ cannot be determined from $\Lambda_{\hat{\varepsilon}}$. 
	
To complete the proof of Theorem \ref{nonuniquenessthm}, let us proceed to the following changes of coordinates
	\begin{equation*}
	\begin{split}
	\mbox{BNCs for $\hat{\varepsilon}^{-1}$} \quad \overset{\tilde{\Psi}}{\longrightarrow} \quad\mbox{BNCs for $(\varepsilon^{-1})_{\flat}$} ,\\
	\mbox{BNCs for $\hat{\varepsilon}'^{-1}$}\quad \overset{\tilde{\Psi}'}{\longrightarrow} \quad \mbox{BNCs for $(\varepsilon'^{-1})_{\flat}$}.
	\end{split}
	\end{equation*}
	Using the relationship between the $\hat{\varepsilon}^{-1}$ and $(\varepsilon^{-1})_{\flat}$ metrics given in \eqref{hatmetrics} as well the fact that $\hat{\varepsilon}$ and $\hat{\varepsilon}'$ have the same expression in their boundary normal coordinates, the Jacobian matrices of $\tilde{\Psi}^{-1}$ and $(\tilde{\Psi}')^{-1}$ are given by, respectively,
	\begin{equation*}
	D\tilde{\Psi}^{-1}=\left(\begin{matrix}
	1&0&0\\
	0&1&0\\
	\tilde{\Psi}_{31}&\tilde{\Psi}_{32}&\left( \frac{\lvert \hat{g}\rvert}{\lvert\hat{\varepsilon}^{-1}\rvert}\right)^{-1/4}
	\end{matrix}\right), \quad 
	D(\tilde{\Psi}')^{-1}=\left(\begin{matrix}
	1&0&0\\
	0&1&0\\
	\tilde{\Psi}'_{31}&\tilde{\Psi}'_{32}&\left( \frac{\lvert\hat{g}'\rvert}{\lvert \hat{\varepsilon}^{-1}\rvert}\right)^{-1/4}
	\end{matrix}\right)
	\end{equation*}
	where the off diagonal entries are undetermined functions which will be zero at the boundary. Denoting by $g_{\varepsilon}$ and $g_{\varepsilon'}$ the metrics $g$ and $g'$ in boundary normal coordinates for $(\varepsilon^{-1})_{\flat}$/$(\varepsilon'^{-1})_{\flat}$ we have
	\begin{equation*}
	\lvert g_{\varepsilon}\rvert= \lvert D\tilde{\Psi}\rvert^{-2}\lvert \hat{g}\rvert= \sqrt{\lvert\hat{g}\rvert}\sqrt{\lvert\hat{\varepsilon}^{-1}\rvert} \quad \mbox{and} \quad
	\lvert g_{\varepsilon'}\rvert= \lvert D\tilde{\Psi}'\rvert ^{-2}\lvert \hat{g}'\rvert= \sqrt{\lvert\hat{g}'\rvert}\sqrt{\lvert\hat{\varepsilon}^{-1}\rvert}.
	\end{equation*}
	Making use of \eqref{gg'hat}, we deduce
	\begin{equation*}
	\lvert g_{\varepsilon}\rvert=\lvert D\Phi\rvert^{-1} \lvert g_{\varepsilon'}\rvert.
	\end{equation*}
	As we showed in Lemma \ref{philemma}, the determinant $\lvert D\Phi \rvert$ can be any positive function which implies that $\lvert g_{\varepsilon} \rvert/\lvert g_{\varepsilon'} \rvert$ can be chosen arbitrarily in a neighbourhood of any point on the boundary and is otherwise equal to one in a neighbourhood of the boundary. By repeating this proof in a neighbourhood of any point on the boundary and combining the results as a composition we can complete the proof for $\epsilon^\sharp$. Demonstrating an approach symmetric to the above we can prove the corresponding result for the determinant $\lvert g \rvert $ in boundary normal coordinates for $\mu^{\sharp}$. This concludes the proof of Theorem \ref{nonuniquenessthm}.
	
	\begin{figure}[h]
		\centering
		\includegraphics[width=0.8\textwidth]{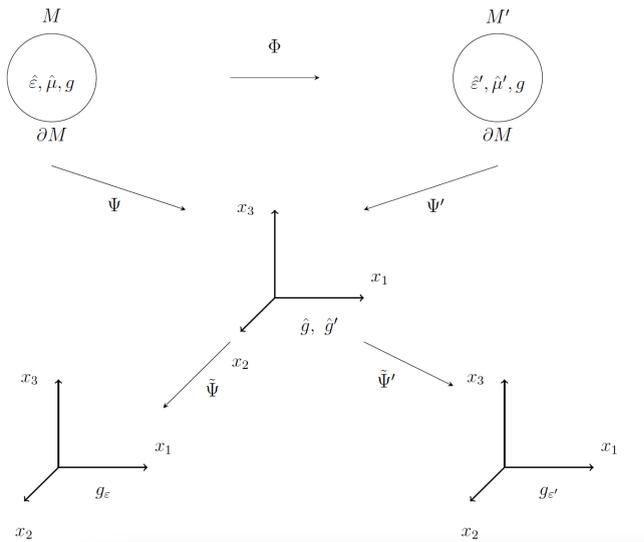}
		\caption{Changes of coordinates used in proof of Theorem \ref{nonuniquenessthm}. Note that the upper part of the diagram commutes (i.e. $\Psi = \Psi' \circ \Phi$).}\label{fig1}
	\end{figure}

	\section{System Decoupling and Factorisation Operators}\label{systemdecouplingsec}
Aiming to decouple the augmented system given in equation \eqref{Peq} for $u_{E}=u_{H}=0$, we decompose $U$ as
	\begin{equation*}
	U=\left(\begin{matrix}U_{H}\\U_{E}\end{matrix}\right),
	\end{equation*}
	where $U_{E}=\left(\begin{matrix} 0\\E\end{matrix}\right)$ and $U_{H}=\left(  \begin{matrix}H\\0\end{matrix}
	\right)$ to obtain
	\begin{equation*}
	\begin{split}
	\left (
	\begin{matrix}
	-i \delta_{\hat{\mu}}& 0\\
	-i \ast_{\hat{\varepsilon}} \mathrm{d}& -i  \mathrm{d}
	\end{matrix}
	\right ) U_H & = - \omega 
	U_E, \quad
	\left (
	\begin{matrix}
	- i \mathrm{d} & i \ast_{\hat{\mu}} \mathrm{d} \\
	0 & -i \delta_{\hat{\varepsilon}}
	\end{matrix}
	\right ) U_E  = - \omega  U_H.
	\end{split}
	\end{equation*}
	Eliminating $U_{E}$ from the above system, we derive the following equation in terms of $H$
	\begin{equation*}
	\left(- \mathrm{d} \delta_{\hat{\mu}}+ \ast_{\hat{\mu}} \mathrm{d} \ast_{\hat{\varepsilon}} \mathrm{d} \right) H - \omega ^{2} H=0.
	\end{equation*}
	which applying the operator $\ast_{\hat{\varepsilon}}\ast_{\hat{\mu}}$ becomes
	\begin{equation*}
L_{H}H=	\left(-\ast_{\hat{\varepsilon}}\ast_{\hat{\mu}} \mathrm{d} \delta_{\hat{\mu}}+ \ast_{\hat{\varepsilon}} \mathrm{d} \ast_{\hat{\varepsilon}} \mathrm{d} \right) H - \omega ^{2} \ast_{\hat{\varepsilon}}\ast_{\hat{\mu}}H=0.
	\end{equation*}
	Or equivalently, in coordinates
	\begin{equation*}
	\begin{split}
	L_{H}H=\frac{\sqrt{\lvert \hat{\mu}^{-1}\rvert}}{\sqrt{\lvert\hat{\varepsilon}^{-1}\rvert}}(\hat{\varepsilon}^{-1})_{lq}\hat{\mu}^{qa}D_{x^{a}}\left(\hat{\mu}^{kb}D_{x^{b}}-i \frac{\partial_{x^{d}}(\sqrt{\lvert \hat{\mu}^{-1}\rvert}\hat{\mu}^{kd})}{\sqrt{\lvert\hat{\mu}^{-1}\rvert}}\right)H_{k}\\-(\hat{\varepsilon}^{-1})_{lq} \frac{\sigma^{aj q}D_{x^{a}}}{\sqrt{\lvert \hat{\varepsilon}^{-1}\rvert }}(\hat{\varepsilon}^{-1})_{bj}\frac{\sigma^{dkb}D_{x^{d}}}{\sqrt{\lvert \hat{\varepsilon}^{-1}\rvert}}H_{k}-\omega^{2} \frac{\sqrt{\lvert \hat{\mu}^{-1}\rvert}}{\sqrt{\lvert \hat{\varepsilon}^{-1}\rvert}}(\hat{\varepsilon}^{-1})_{lq}\hat{\mu}^{qk}H_{k}=0.
	\end{split}
	\end{equation*}
	Following the approach of \cite{Nakamura1997layer}, the next step is to separate the differentiation with respect to $x_{3}$ from the differentiation with respect to $\tilde{x}$ to write the differential operator $L_{H}$ as
	\begin{equation*}
	L_{H}=T_{H}(x)D_{x_{3}}^{2}+A_{H}(x,D_{\tilde{x}}) D_{x_{3}}+ G_{H}(x)D_{x_{3}}+Q_{H}(x,D_{\tilde{x}})+ F_{H}(x,D_{\tilde{x}}) +R_{H}(x),
	\end{equation*}
where the symbols of the operators in the representation of $L_{H}$ are given by
\begin{equation} \label{Tsym}
\begin{split}
(T_{H}(x))_{l}^{k}&=\frac{\sqrt{\lvert \hat{\mu}^{-1}\rvert}}{\sqrt{\lvert \hat{\varepsilon}^{-1}\rvert}}(\hat{\varepsilon}^{-1})_{lq} \hat{\mu}^{q3}\hat{\mu}^{k3}+\left(\delta_{l}^{k}\hat{\varepsilon}^{33}-\delta_{l}^{3}\hat{\varepsilon}^{3k}   \right),\\
\end{split}
\end{equation}
\begin{equation} \label{Asym}
\begin{split}
(A_{H}(x,\tilde{\xi}))_{l}^{k}&= \frac{\sqrt{\lvert \hat{\mu}^{-1}\rvert}}{\sqrt{\lvert \hat{\varepsilon}^{-1}\rvert}}(\hat{\varepsilon}^{-1})_{lq}\left(
\hat{\mu}^{q3}\hat{\mu}^{k\tilde{a}}\xi_{\tilde{a}}+ \hat{\mu}^{q\tilde{a}}\xi_{\tilde{a}}\hat{\mu}^{k3}\right)\\
&\hskip3cm+ 2 \delta_{l}^{k}\langle \nu, \tilde{\xi}\rangle_{\hat{\varepsilon}}- \delta_{l}^{3}\xi_{\tilde{a}}\hat{\varepsilon}^{\tilde{a}k} - \delta_{l}^{\tilde{l}}\xi_{\tilde{l}} \hat{\varepsilon}^{3k} ,
\end{split}
\end{equation}
\begin{equation} \label{Gsym}
\begin{split}
(G_{H}(x))_{l}^{k}&=
-i\frac{\sqrt{\lvert \hat{\mu}^{-1}\rvert }}{\sqrt{\lvert \hat{\varepsilon}^{-1}\rvert }}(\hat{\varepsilon}^{-1})_{lq}\left(\hat{\mu}^{qa}\partial_{x^{a}}\hat{\mu}^{k3} +  \hat{\mu}^{q3}\frac{\partial_{x^{a}}(\sqrt{\lvert \hat{\mu}^{-1}\rvert }\hat{\mu}^{ka})}{\sqrt{\lvert\hat{\mu}^{-1}\rvert}}\right) \\&+ i (\hat{\varepsilon}^{-1})_{lq}\frac{\sigma^{aj q}}{\sqrt{\lvert\hat{\varepsilon}^{-1}\rvert}}\partial_{x^{a}}(\hat{\varepsilon}^{-1})_{\tilde{b}j}\frac{\sigma^{3\tilde{k}\tilde{b}}}{\sqrt{\lvert \hat{\varepsilon}^{-1}\rvert}},
\end{split}
\end{equation}
\begin{equation}\label{Qsym}
\begin{split}
(Q_{H}(x,\tilde{\xi}))_{l}^{k}&=\frac{\sqrt{\lvert\hat{\mu}^{-1}\rvert}}{\sqrt{\lvert \hat{\varepsilon}^{-1}\rvert}}
(\hat{\varepsilon}^{-1})_{lq}\hat{\mu}^{q\tilde{a}}\xi_{\tilde{a}}\hat{\mu}^{k\tilde{b}}\xi_{\tilde{b}}+ \left(\delta_{l}^{k}\lvert\tilde{\xi}\rvert_{\hat{\varepsilon}}^{2}-\delta_{l}^{\tilde{l}}\xi_{\tilde{l}}\xi_{\tilde{a}}\hat{\varepsilon}^{\tilde{a}k}   \right),
\end{split}
\end{equation}
\begin{equation}\label{Fsym}
\begin{split}
(F_{H}(x,\tilde{\xi}))_{l}^{k}&=
-i\frac{\sqrt{\lvert \hat{\mu}^{-1}\rvert}}{\sqrt{\lvert \hat{\varepsilon}^{-1}\rvert}}(\hat{\varepsilon}^{-1})_{lq}\left(\hat{\mu}^{qa}\partial_{x^{a}}\hat{\mu}^{k\tilde{a}}\xi_{\tilde{a}} + \hat{\mu}^{q\tilde{d}}\xi_{\tilde{d}}\frac{\partial_{x^{a}}(\sqrt{\lvert\hat{\mu}^{-1}\rvert}\hat{\mu}^{ka})}{\sqrt{\lvert \hat{\mu}^{-1}\rvert}} \right)\\&+ i (\hat{\varepsilon}^{-1})_{lq}\frac{\sigma^{aj q}}{\sqrt{\lvert \hat{\varepsilon}^{-1}\rvert}}\partial_{x^{a}}(\hat{\varepsilon}^{-1})_{bj}\frac{\sigma^{\tilde{d}kb}}{\sqrt{\lvert \hat{\varepsilon}^{-1}\rvert}}\xi_{\tilde{d}},
\end{split}
\end{equation}
\begin{equation*}
(R_{H}(x))_{l}^{k}= - \omega^{2}\frac{\sqrt{\lvert \hat{\mu}^{-1}\rvert}}{\sqrt{\lvert\hat{\varepsilon}^{-1}\rvert}}(\hat{\varepsilon}^{-1})_{lq}\hat{\mu}^{qk}- (\hat{\varepsilon}^{-1})_{lq}\hat{\mu}^{qb}\partial_{x^{b}}\frac{\partial_{x^{a}}(\sqrt{\lvert\hat{\mu}^{-1}\rvert }\hat{\mu}^{ka})}{\sqrt{\lvert\hat{\mu}^{-1}\rvert}}. 
\end{equation*}
Next we show that $L_H$ can be factorised using pseudodifferential operators. We will make use of the principal symbol of $L_H$ which is
\[
\sigma_p(L_H) = M_{H}(\xi_{3})=T_{H}(x)\xi_{3}^{2}+ A_{H}(x,\tilde{\xi})\xi_{3}+Q_{H}(x,\tilde{\xi}).
\]
$M_H$ is a matrix polynomial with respect to the variable $\xi_3$ in the sense of \cite{Lancaster} and, for fixed $(x,\tilde{\xi})$, the eigenvalues are the solutions of the sixth order polynomial equation $\mathrm{det}(M_{H}(\xi_{3})) = 0$. 
\begin{theorem}\label{thm:LHfactor}
Suppose that we have some local boundary coordinates $\{x^j\}_{j=1}^3$ and let $\Gamma_+(x,\tilde{\xi}) \subset \mathbb{C}$ be a contour in $\mathbb{C}$ enclosing all eigenvlaues of $M_H$ with positive imaginary part. Then, in the domain of the coordinates, there exist pseudodifferential operators $B(x,D_{\tilde{x}})$ and $G_{0}(x, D_{\tilde{x}})$ of order $1$ and $0$, respectively, such that the principal symbol of $B$ is given by
\begin{equation}\label{Bcont}
B^{(1)}(x,\tilde{\xi})= \int_{\Gamma^{+}} \xi_{3} M_{H}(\xi_{3})^{-1} \ d\xi_{3} \left(\int_{\Gamma^{+}}  M_{H}(\xi_{3})^{-1} \ d\xi_{3}\right)^{-1}
\end{equation}
and
\begin{equation}\label{factorizationLH}
	L_{H}= \left(\mathbb{I}D_{x^{3}}- B^{\ast}(x,D_{\tilde{x}})+ G_{0} (x,D_{\tilde{x}}) \right)T_{H}(x)\left(\mathbb{I} D_{x^{3}}- B(x,D_{\tilde{x}}) \right)
\end{equation}
modulo smoothing. Additionally, $B(x,D_{\tilde{x}})$ satisfies the following Riccati equation
\begin{equation}
T_{H}D_{x^{3}}B+ (A_{H}+ G_{H})B+ T_{H}B^{2}+ Q_{H}+ F_{H} + R_{H}=0 \label{RiccatiB}
\end{equation}
and the symbol of $B$ is a smooth function of $\hat{\varepsilon}$ and $\hat{\mu}$.
\end{theorem}

\begin{proof}
The proof is the same as the proof of \cite[Theorem 2.2]{Nakamura1997layer} which we briefly sketch. First, the Riccati equation \eqref{RiccatiB} can be derived from \eqref{factorizationLH} and the expression of the operator $L_{H}$ as done in the proof of \cite[Theorem 2.2]{Nakamura1997layer}. Looking at the principal symbol of the operators on both sides of \eqref{RiccatiB}, we see that the principal symbol $B^{(1)}$ of $B$ must satisfy
\begin{equation*}
T_{H}(x)(B^{(1)}(x,\tilde{\xi}))^{2}+ A_{H}(x,\tilde{\xi})B^{(1)}(x,\tilde{\xi})+Q_{H}(x,\tilde{\xi})=0
\end{equation*}
and using \cite[Theorem 4.2]{Lancaster} we see that $B^{(1)}$ given by \eqref{Bcont} will satisfy this equation. The remainder of the symbol of $B$ can then be determined by looking at an asymptotic symbol expansion of \eqref{RiccatiB} and this symbol can be asymptotically summed to produce a pseudodifferential operator that satisfies \eqref{RiccatiB} modulo smoothing. Then the symbol for $G_0$ can be determined from \eqref{factorizationLH} in the same way.

The final statement about smoothness of the symbol of $B$ with respect to the parameters $\hat{\varepsilon}$ and $\hat{\mu}$ can be proven by first looking at the principal symbol given by \eqref{Bcont}. Indeed, given $(x,\tilde{\xi})$ a single choice of $\Gamma_+$ will work also in a neighbourhood of $(x,\tilde{\xi})$, and so the regularity with respect to the parameters follows from the regularity of $M_H(\xi_3)^{-1}$ with respect to the same on the contour $\Gamma_+$. The regularity of the rest of the symbol with respect to $\hat{\varepsilon}$ and $\hat{\mu}$ follows from this and the construction of the lower order parts of the symbol as given in \cite{Nakamura1997layer}.
\end{proof}

Next we express the principal symbol $B^{(1)}$ given by \eqref{Bcont} more explicitly. The expression we derive will be used in the calculation of the principal symbols of $\Lambda_{\hat{\varepsilon}}$ and $\Lambda_{\hat{\mu}}$ later.
Our method uses the theory of monic matrix polynomials \cite{Lancaster} to determine the Jordan pairs of $M_{H}(\xi_{3})$ and then express $B^{(1)}(x,\tilde{\xi})$ in terms of them. The theory of monic matrix polynomials is applicable to $M_{H}(\xi_{3})$ as the coefficient $T_{H}(x)$ is invertible.

In order to calculate the Jordan pairs of $M_{H}$, we should find the eigenvalues of $M_H$ and for each eigenvalue $\xi_3$ also find covectors $y_j$ such that 
	\begin{equation*}\label{MHcordinv}
	\sum_{p=0}^{j} \frac{1}{p!}\partial_{\xi_{3}}^{p} M_{H}(\xi_{3})y_{j-p}=0, \quad j=0, \dots k-1
	\end{equation*}
where $k$ is at most the order of $\xi_{3}$ as a root of the equation $\det(M_{H}(\xi_{3}))$=0. For each eigenvalue, the covector $y_{0}$ is called a corresponding eigenvector and the covectors $y_{1},\dots y_{k-1} $ are called generalised eigenvectors. Writing $M_{H}(\xi_{3})$ in coordinates, we have
\begin{equation*}
\left(\frac{\sqrt{\lvert \hat{\mu}^{-1}\rvert}}{\sqrt{\lvert\hat{\varepsilon}^{-1}\rvert}}(\hat{\varepsilon}^{-1})_{lq}\hat{\mu}^{qa}\xi_{a}\hat{\mu}^{kd}\xi_{d}+ \delta_{l}^{k}\lvert \xi\rvert_{\hat{\varepsilon}}^{2}- \xi_{l}\xi_{d}\hat{\varepsilon}^{dk}\right) H_{k}=0,
\end{equation*}
where we have made use of the identity \eqref{identitycofact} for the $\hat{\varepsilon}^{-1}$ metric. Let us now consider the covectors
\begin{equation}\label{cordinvcovectors}
\xi, \quad \zeta=\hat{\varepsilon}^{-1}\hat{\mu}\xi, \quad \chi= \ast_{\hat{\varepsilon}} (\xi \wedge \zeta).
\end{equation}
When these covectors form a basis (see remark \ref{basisremark}), we can write $M_{H}$ in the basis $\{\chi,\xi,\zeta\}$ as follows
\begin{equation*}
[\begin{matrix} \chi & \xi&\zeta\end{matrix}]^{T}\hat{\varepsilon}{M_{H}} [\begin{matrix} \chi & \xi&\zeta\end{matrix}]=\left(\begin{matrix}
\lvert \xi\rvert_{\hat{\varepsilon}}^{2}&0&0\\
0&0&\frac{\sqrt{\lvert \hat{\mu}^{-1}\rvert}}{\sqrt{\lvert\hat{\varepsilon}^{-1}\rvert}}\lvert\xi\rvert_{\hat{\mu}}^{2}\\
0& -\lvert\xi\rvert_{\hat{\mu}}^{2}&\frac{\sqrt{\lvert\hat{\mu}^{-1}\rvert}}{\sqrt{\lvert\hat{\varepsilon}^{-1}\rvert}}\lvert\xi\rvert_{\hat{\mu}\hat{\varepsilon}^{-1}\hat{\mu}}^{2}+\lvert\xi\rvert_{\hat{\varepsilon}}^{2}
\end{matrix}
\right).
\end{equation*}
The determinant of the above vanishes for $\xi_{3}$ satisfying
\begin{equation}\label{eigenvalue}
\lvert \xi\rvert_{\hat{\mu}}^{4} \lvert \xi\rvert_{\hat{\varepsilon}}^{2}=0.
\end{equation}
\begin{remark} \label{innervanish}
Recalling Remark \ref{innerproductremark}, let us point out that the quantities $\lvert \xi\rvert_{\hat{\mu}}^2$ and $ \lvert \xi\rvert_{\hat{\varepsilon}}^2$ can vanish for complex $\xi$ as they are given by the real inner products $\langle \cdot, \cdot \rangle_{\hat{\varepsilon}}$ $\langle \cdot, \cdot \rangle_{\hat{\mu}}$. 
\end{remark}
For fixed $\tilde{\xi}$, the solutions $\xi_{3}$ of \eqref{eigenvalue} are conjugate pairs denoted by ${\xi_{\hat{\mu}}}_{3}$, $\overline{\xi_{\hat{\mu}}}_{3}$ and ${\xi_{\hat{\varepsilon}}}_{3}$, $\overline{\xi_{\hat{\varepsilon}}}_{3}$, respectively. The values of $\xi_{3}$ with positive imaginary part denoted by ${\xi_{\hat{\mu}}}_{3}$ and ${\xi_{\hat{\varepsilon}}}_{3}$ are given by, respectively
\begin{equation}\label{xihatmu3}
{\xi_{\hat{\mu}}}_{3}=- \frac{\langle \nu,\tilde{\xi}\rangle_{\hat{\mu}}}{\hat{\mu}^{33}}+ i \frac{\lvert\tilde{\xi}\rvert_{\tilde{\mu}}}{\sqrt{\hat{\mu}^{33}}},
\end{equation}
\begin{equation}\label{xihatepsilon3}
{\xi_{\hat{\varepsilon}}}_{3}=- \frac{\langle \nu,\tilde{\xi}\rangle_{\hat{\varepsilon}}}{\hat{\varepsilon}^{33}}+ i \frac{\lvert\tilde{\xi}\rvert_{\tilde{\varepsilon}}}{\sqrt{\hat{\varepsilon}^{33}}},
\end{equation}
where we have made use of the representations of $\hat{\varepsilon}$ and $\hat{\mu}$ given in equations \eqref{epsilonepsilon'tilde}, \eqref{mumu'tilde}. We notice that ${\xi_{\hat{\mu}}}_{3}$ is a double root of $\det(M_{H}(\xi_{3}))$ and we expect an eigenvector and a generalised eigenvector associated with this value of $\xi_{3}$.
The eigenvector corresponding to ${\xi_{\hat{\mu}}}_{3}$ will be denoted by $\xi_{\hat{\mu}}$ and is given by
\begin{equation*}
\xi_{\hat{\mu}}= \xi \vert_{{\xi_{3}=\xi_{\hat{\mu}}}_{3}}.
\end{equation*}
 The generalised eigenvector is given by
	\begin{equation*}
\gamma_{\hat{\mu}}= \nu + m_{\hat{\mu}} \zeta_{\hat{\mu}},
\end{equation*}
where
\begin{equation*}
\zeta_{\hat{\mu}}=\zeta\vert_{{\xi_{3}=\xi_{\hat{\mu}}}_{3}} \quad \mbox{and} \quad     m_{\hat{\mu}}=  -  \frac{2\frac{\sqrt{\lvert \hat{\mu}^{-1}\rvert}}{\sqrt{\lvert\hat{\varepsilon}^{-1}\rvert}}\langle \nu,\xi_{\hat{\mu}}\rangle_{\hat{\mu}}}{\frac{\sqrt{\lvert\hat{\mu}^{-1}\rvert}}{\sqrt{\lvert\hat{\varepsilon}^{-1}\rvert}}\lvert\xi_{\hat{\mu}}\rvert_{\hat{\mu}\hat{\varepsilon}^{-1}\hat{\mu}}^{2}+ \lvert\xi_{\hat{\mu}}\rvert_{\hat{\varepsilon}}^{2}}.
\end{equation*}
Proceeding to the eigenvalue ${\xi_{\hat{\varepsilon}}}_{3}$, as it is a single root of $\det(M_{H}(\xi_{3}))$, it corresponds to the eigenvector
$\chi_{\hat{\varepsilon}}$, defined by
\begin{equation*}
\chi_{\hat{\varepsilon}}= \chi \vert_{\xi_3 ={\xi_{\hat{\varepsilon}}}_{3}}.
\end{equation*}
Therefore, the Jordan pairs and blocks of the matrix polynomial $M_{H}(\xi_{3})$ are given by
\begin{equation*}
Y_{\hat{\varepsilon}}= \chi_{\hat{\varepsilon}} , \quad J_{\hat{\varepsilon}}=[{\xi_{\hat{\varepsilon}}}_{3}], \quad \mbox{and} \quad 	Y_{\hat{\mu}}= \left[\begin{matrix}\xi_{\hat{\mu}}&\gamma_{\hat{\mu}} \end{matrix} \right], \quad J_{\hat{\mu}}= \left( \begin{matrix}
{\xi_{\hat{\mu}}}_{3}&1\\
0&{\xi_{\hat{\mu}}}_{3}
\end{matrix} \right).
\end{equation*}
Using the Jordan pairs in the above, the principal symbol of the pseudodifferential operator $B(x,D_{\tilde{x}})$ is expressed as
\begin{equation}\label{BJordan}
B^{(1)}(x,\tilde{\xi})= X_{H}J_{H}X_{H}^{-1},
\end{equation}
where
\begin{equation*}
J_{H}= \left(\begin{matrix}
{\xi_{\hat{\varepsilon}}}_{3}&0&0\\
0&{\xi_{\hat{\mu}}}_{3}&\lvert\xi_{\hat{\mu}}\rvert_{\hat{\varepsilon}}\\
0&0&{\xi_{\hat{\mu}}}_{3}
\end{matrix}
\right)\quad \mbox{and} \quad
X_{H}= \left[\begin{matrix}
\hat{\chi}_{\hat{\varepsilon}}& \hat{\xi}_{\hat{\mu}}& \gamma_{\hat{\mu}}
\end{matrix}\right]
\end{equation*}
and $\hat{\chi}_{\hat{\varepsilon}}$, $\hat{\xi}_{\hat{\mu}}$ denote that the covectors $\chi_{\hat{\varepsilon}}$ and $\xi_{\hat{\mu}}$ are normalised in such a way that $X_{H}$ $\in$ $\Psi DO^{(0)}$.
Additionally, $M_{H}(\xi_{3})$ can be factorised in terms of $B^{(1)}(x,\tilde{\xi})$ as 
\begin{equation}\label{Mfactorization}
M_{H}(\xi_{3})=\left(\mathbb{I}\xi_{3}-(B^{\ast})^{(1)}(x,\tilde{\xi})\right)T_{H}(x)\left(\mathbb{I}\xi_{3} -B^{(1)}(x,\tilde{\xi})\right),
\end{equation}
where in order to derive the above factorisation we use the fact that 	$T_{H}^{1/2}B^{(1)}T_{H}^{-1/2}$, $T_{H}^{1/2}(B^{\ast})^{(1)}T_{H}^{-1/2}$ are spectral right and left divisors of $T_{H}^{-1/2}M_{H}(\xi_{3})T_{H}^{-1/2}$, respectively. Factorising $M_{H}(\xi_{3})$, we can follow the method of \cite{Nakamura1997layer} to obtain the factorisation of $L_{H}$ provided in \eqref{factorizationLH}, where the principal symbol of $B(x,D_{\tilde{x}})$ is $B^{(1)}(x,\tilde{\xi})$.
\begin{remark}\label{basisremark}
	We observe that when ${\xi_{\hat{\varepsilon}}}_{3}={\xi_{\hat{\mu}}}_{3}$, the covectors $\xi$, $\chi$ and $\zeta$ do not constitute a basis and the representation \eqref{BJordan} does not apply. However, using the continuity of the symbol $B(x,\tilde{\xi})$ with respect to $\hat{\varepsilon}$ and $\hat{\mu}$ the various symbol expressions for the impedance and other maps derived below are still valid even when ${\xi_{\hat{\varepsilon}}}_{3}={\xi_{\hat{\mu}}}_{3}$.
\end{remark}

Next, we state the following result which will be used in the calculation of the principal symbol of the impedance map. The proof follows the proof of \cite[Proposition 2]{Mcdowall1997}.
\begin{theorem}\label{Dx3Htheorem}
	If $H$ is the solution of the Dirichlet problem consisting of \eqref{Maxhat}, then
	\begin{equation*}
	D_{x_{3}}H=B(x,D_{\tilde{x}})H,\ \mbox{modulo smoothing}, \ \mbox{in a set of boundary coordinates}\ \mbox{in} \  M.
	\end{equation*}
\end{theorem}


The corresponding factorisation operator obtained by decoupling the system in terms of the electric field $E$ and following similar reasoning to the above, will be denoted by $C(x,D_{\tilde{x}})$. The pseudodifferential operator $C(x,D_{\tilde{x}})$ satisfies the Riccati equation
\begin{equation}
T_{E}D_{x^{3}}C+ (A_{E}+ G_{E})C+ T_{E}C^{2}+ Q_{E}+ F_{E} + R_{E}=0, \label{RiccatiC}
\end{equation}
where $T_{E}$, $A_{E}$, $G_{E}$, $Q_{E}$, $F_{E}$, $R_{E}$ are obtained from $T_{H}$, $A_{H}$, $G_{H}$, $Q_{H}$, $F_{H}$, $R_{H}$ by exchanging $\hat{\varepsilon}$ with $\hat{\mu}$.
The principal symbol of $C(x,D_{\tilde{x}})$ is provided by
\begin{equation} \label{CJordan}
C^{(1)}(x,\tilde{\xi})= X_{E}J_{E}X_{E}^{-1},
\end{equation}
where
\begin{equation*}
J_{E}= \left(\begin{matrix}
{\xi_{\hat{\mu}}}_{3}&0&0\\
0&{\xi_{\hat{\varepsilon}}}_{3}&\lvert \xi_{\hat{\varepsilon}}\rvert_{\hat{\mu}}\\
0&0&{\xi_{\hat{\varepsilon}}}_{3}
\end{matrix}
\right)\quad \mbox{and} \quad
X_{E}= \left[\begin{matrix}
\hat{\chi}_{\hat{\mu}}& \hat{\xi}_{\hat{\varepsilon}}& \gamma_{\hat{\varepsilon}}
\end{matrix}\right]
\end{equation*}
and the covectors $\chi_{\hat{\mu}}$, $\xi_{\hat{\varepsilon}}$ and $\gamma_{\hat{\varepsilon}}$ are obtained from $\chi_{\hat{\varepsilon}}$, $\xi_{\hat{\mu}}$ and $\gamma_{\hat{\mu}}$ by exchanging $\hat{\mu}$ with $\hat{\varepsilon}$.

Viewing the solution $H$ of the Dirichlet problem involving equations \eqref{Maxhat} and \eqref{pullbackBCs} as a pseudodifferential operator acting on the boundary data $F$ and motivated by the expression of $B^{(1)}(x,\tilde{\xi})$, the principal symbol of $H$ in the basis $\{\chi_{\hat{\varepsilon}},\xi_{\hat{\mu}},\gamma_{\hat{\mu}}\}$ is written as
\begin{equation} \label{H0decomp}
H_{k}^{(0)}= a_{\hat{\varepsilon}}{{\chi}_{\hat{\varepsilon}}}_{k}+ b_{\hat{\varepsilon}}{{\xi}_{\hat{\mu}}}_{k},
\end{equation}
	where
\begin{equation}\label{abepsilonconst}
a_{\hat{\varepsilon}}= \frac{\ast_{\iota^{\ast}\hat{\varepsilon}}(\tilde{\xi}\wedge F)}{\langle \nu_{\hat{\varepsilon}},\xi_{\hat{\varepsilon}}\rangle_{\hat{\varepsilon}} \lvert\xi_{\hat{\varepsilon}}\rvert_{\hat{\mu}}^{2}}, \quad b_{\hat{\varepsilon}}= \frac{\ast_{\iota^{\ast}\hat{\varepsilon}}(F\wedge\chi_{\hat{\varepsilon}})}{\langle \nu_\epsilon,\xi_{\hat{\varepsilon}}\rangle_{\hat{\varepsilon}} \lvert\xi_{\hat{\varepsilon}}\rvert_{\hat{\mu}}^{2}}.
\end{equation}
Here $\ast_{\iota^{\ast}\hat{\varepsilon}}$ denotes the Hodge star operator with respect to the $\hat{\varepsilon}$ metric pulled back to $\partial M$ and $\nu_{\hat{\varepsilon}}$ is the normal covector at $\partial M$ with respect to the $\hat{\varepsilon}$ metric. This decomposition of the principal symbol of $H$ is implied since it has to vanish on the principal symbol of the divergence equation $\xi_{\tilde{a}}\hat{\mu}^{\tilde{a}b}H_{b}^{(0)}+ \hat{\mu}^{3a}(B^{(1)}(x,\tilde{\xi}))_{a}^{b}H_{b}^{(0)}=0$. Similarly, the principal symbol of $E$ seen as a pseudodifferential operator acting on the boundary data $G$ is given by
	\begin{equation*}
E_{k}^{(0)}= a_{\hat{\mu}}{{\chi}_{\hat{\mu}}}_{k}+ b_{\hat{\mu}}{{\xi}_{\hat{\varepsilon}}}_{k},
\end{equation*}
where $a_{\hat{\mu}}$ and $b_{\hat{\mu}}$ are obtained from $a_{\hat{\varepsilon}}$, $b_{\hat{\varepsilon}}$ by exchanging $\hat{\varepsilon}$ with $\hat{\mu}$ and replacing $F$ by $G$.

	\section{Principal Symbols of Boundary Mappings and Boundary Recovery}\label{boundaryrecoverysec}
	
	\begin{theorem}\label{impedanceprincsymbol}
		The principal symbol of the impedance map $\Lambda_{\hat{\varepsilon}}: \iota^{\ast}H \mapsto \iota^{\ast}E$ is given by
		\begin{equation*}
		\lambda_{\hat{\varepsilon}}^{(1)}(F) = -\frac{\tilde{\xi} \ \ast_{\iota^{\ast}\hat{\varepsilon}}(\tilde{\xi}\wedge F  ) }{\omega \langle \nu_{\hat{\varepsilon}}, \xi_{\hat{\varepsilon}}\rangle_{\hat{\varepsilon}}}.
		\end{equation*}
	\end{theorem}
	\begin{proof}
		Let us consider the Maxwell's equation
		\begin{equation*}
		\ast_{\hat{\varepsilon}} \mathrm{d}H= -i \omega E,
		\end{equation*}
		which taking the pullback and considering the principal symbol in coordinates gives
		\begin{equation*}
	\lambda_{\hat{\varepsilon}}^{(1)}(F)	 =-\frac{1}{\omega}(\hat{\varepsilon}^{-1})_{\tilde{l}b}\frac{\sigma^{\tilde{a}kb}}{\sqrt{\lvert\hat{\varepsilon}^{-1}\rvert}} \xi_{\tilde{a}}H_{k}^{(0)} -\frac{1}{\omega}(\hat{\varepsilon}^{-1})_{\tilde{l}\tilde{b}}\frac{\sigma^{3\tilde{a}\tilde{b}}}{\sqrt{\lvert\hat{\varepsilon}^{-1}\rvert}}(B^{(1)}(x,\tilde{\xi}))_{\tilde{a}}^{k} H_{k}^{(0)}. \label{curlHsymb}
		\end{equation*}
Observing that $\xi_{\hat{\mu}}$ in the decomposition \eqref{H0decomp} of $H^{(0)}$ is in the kernel of the right hand-side of the above equation, only the covector $\chi_{\hat{\varepsilon}}$ contributes to the principal symbol of $\Lambda_{\hat{\varepsilon}}$
		as follows
		\begin{equation*}
	\lambda_{\hat{\varepsilon}}^{(1)}(F) = - a_{\hat{\varepsilon}}\frac{1}{\omega } (\hat{\varepsilon}^{-1})_{\tilde{l}b} \frac{\sigma^{akb}}{\sqrt{\lvert\hat{\varepsilon}^{-1}\rvert}}{\xi_{\hat{\varepsilon}}}_{a}{\chi_{\hat{\varepsilon}}}_{k}= - \frac{a_{\hat{\varepsilon}}}{\omega }{\xi}_{\tilde{l}}\lvert\xi_{\hat{\varepsilon}}\rvert_{\hat{\mu}}^{2},
		\end{equation*}
		where we have used the identity \eqref{identitycofact} for the $\hat{\varepsilon}^{-1}$ metric.
		Substituting the expression \eqref{abepsilonconst} for $a_{\hat{\varepsilon}}$ we obtain the result.
	\end{proof}
Working similarly to the above we can calculate the principal symbol of the admittance map as stated in the following theorem.
	\begin{theorem}\label{admittanceprincsymbol}
	The principal symbol of the admittance map $\Lambda_{\hat{\mu}}: \iota^{\ast}E \mapsto \iota^{\ast}H$ is given by
	\begin{equation*}
	\lambda_{\hat{\mu}}^{(1)}(G)= \frac{\tilde{\xi} \ \ast_{\iota^{\ast}\hat{\mu}}(\tilde{\xi}\wedge G )  }{\omega\langle \nu_{\hat{\mu}}, \xi_{\hat{\mu}}\rangle_{\hat{\mu}}}.
	\end{equation*}
\end{theorem}

\subsection{Proof of Theorem \ref{metricsequalinbncsthm}} \label{metricsequalinbncsec}

From the fact that $\lambda_{\hat{\varepsilon}}^{(1)}=\lambda_{\hat{\varepsilon}'}^{(1)}$  and using Theorem \ref{impedanceprincsymbol}, we get
	\begin{equation*}
	\sqrt{\lvert\tilde{\varepsilon}^{-1}\rvert}\langle \nu_{\hat{\varepsilon}},\xi_{\hat{\varepsilon}}\rangle_{\hat{\varepsilon}}	=\sqrt{\lvert\tilde{\varepsilon}'^{-1}\rvert}\langle \nu_{\hat{\varepsilon}'},\xi_{\hat{\varepsilon}'}\rangle_{\hat{\varepsilon}'}, \quad \mbox{at} \ x_{3}=0,
	\end{equation*}
which, from \eqref{xihatepsilon3}, leads to
	\begin{equation*}
	\sqrt{\lvert\tilde{\varepsilon}^{-1}\rvert} \lvert \tilde{\xi}\rvert_{\tilde{\varepsilon}} = \sqrt{\lvert\tilde{\varepsilon}'^{-1}\rvert} \lvert \tilde{\xi}\rvert_{\tilde{\varepsilon}'}, \quad \mbox{at} \ x_{3}=0.
	\end{equation*}
	This implies that
	\begin{equation}
	\tilde{\varepsilon}=\tilde{\varepsilon}' \quad \mbox{at} \ x_{3}=0.
	\end{equation}
	 Working in the same way for $\lambda_{\hat{\mu}}^{(1)}=\lambda_{\hat{\mu}'}^{(1)}$ we reach the corresponding result for $\tilde{\mu}$ and $\tilde{\mu}'$ which completes the proof of Theorem \ref{metricsequalinbncsthm}.
	 
		The result of Theorem \ref{metricsequalinbncsthm} in terms of the metrics $(\varepsilon^{\sharp},\mu^{\sharp})$, $(\varepsilon'^{\sharp},\mu'^{\sharp})$ translates to
		\begin{equation*}
		\lvert g \rvert	\varepsilon^{\tilde{i}\tilde{j}}=\lvert g'\rvert \varepsilon'^{\tilde{i}\tilde{j}}, \quad \mbox{at} \ x_{3}=0,
		\end{equation*}
		in boundary normal coordinates for $\varepsilon^{\sharp}$/ $\varepsilon'^{\sharp}$ and
			\begin{equation*}
		\lvert g \rvert	\mu^{\tilde{i}\tilde{j}}=\lvert g'\rvert \mu'^{\tilde{i}\tilde{j}}, \quad \mbox{at} \ x_{3}=0,
		\end{equation*}
		in boundary normal coordinates for $\mu^{\sharp}$/ $\mu'^{\sharp}$. Theorem \ref{nonuniquenessthm} implies that the presence of the determinant $\lvert g \rvert $ in the above equations leads to a non-uniqueness on the boundary recovery of $\varepsilon^{\sharp}$ and $\mu^{\sharp}$. More specifically, we have that the tangential components of the metrics $\varepsilon^{\sharp}$, $\mu^{\sharp}$ can be determined up to a conformal factor which is arbitrary.
\subsection{Further consequences}
Now let us investigate the consequences of Theorem \ref{metricsequalinbncsthm} to the different pairs of electromagnetic fields $E$, $H$ and $E'$, $H'$.
\begin{theorem}\label{fieldsequalattheboundary}
	Let $(\hat{\varepsilon},\hat{\mu})$, $(\hat{\varepsilon}',\hat{\mu}')$  be two different metrics and $\Lambda_{\hat{\varepsilon}}$, $\Lambda_{\hat{\varepsilon}'}$, $E$, $H$, $E'$, $H'$ the associated impedance maps and electromagnetic fields, respectively. Assuming that $\Lambda_{\hat{\varepsilon}}=\Lambda_{\hat{\varepsilon}'}$ and $\iota^{\ast}H=\iota^{\ast}H'$ then
	\begin{equation*}
	E=E', \quad \mbox{at} \ x_{3}=0,
	\end{equation*}
	in boundary normal coordinates for $\hat{\varepsilon}$/ $\hat{\varepsilon}'$ and
	\begin{equation*}
	H=H', \quad \mbox{at} \ x_{3}=0,
	\end{equation*}
	in boundary normal coordinates for $\hat{\mu}$/$\hat{\mu}'$.
\end{theorem}


\begin{proof}
We will work with $E$, $E'$ and the result for $H$, $H'$ is derived accordingly. Since $\iota^{\ast}E=\iota^{\ast}E'$, it suffices to prove that $E_{3}=E_{3}'$ at $x_{3}=0$.
	To start with, let us consider the third component of the equation
	\begin{equation*}
	\ast_{\hat{\varepsilon}} \mathrm{d} H= -i \omega E, \quad \mbox{at} \ x_{3}=0,
	\end{equation*}
in boundary normal coordinates for $\hat{\varepsilon}$, which is
	\begin{equation*}
	E_{3}=	\frac{i}{\omega}\frac{\sigma^{\tilde{a}\tilde{k}3}}{\sqrt{\lvert\tilde{\varepsilon}^{-1}\rvert}} \frac{\partial H_{\tilde{k}}}{\partial x_{\tilde{a}}}, \quad \mbox{at} \ x_{3}=0.
	\end{equation*}
	Similarly, the corresponding equation for the $'-$ parameters in boundary normal coordinates for $\hat{\varepsilon}'$ is given by
	\begin{equation*}
	E'_{3}=	\frac{i}{\omega}\frac{\sigma^{\tilde{a}\tilde{k}3}}{\sqrt{\lvert\tilde{\varepsilon}^{-1}\rvert}} \frac{\partial H_{\tilde{k}}}{\partial x_{\tilde{a}}}, \quad \mbox{at} \ x_{3}=0,
	\end{equation*}
	where we have used the result of Theorem \ref{metricsequalinbncsthm} and the assumption $\iota^{\ast}H=\iota^{\ast}H'$.
	Therefore, $E_{3}=E'_{3}$ at $x_{3}=0$ which completes the proof.
\end{proof}
Let $\Psi$, $\Psi'$ be the following change of coordinates transformations
\begin{equation*}
\mbox{BNCs for } \hat{\varepsilon} \quad \overset{\Psi} {\longrightarrow} \quad  \mbox{BNCs for }\hat{\mu},
\end{equation*}
\begin{equation*}
\mbox{BNCs for } \hat{\varepsilon}' \quad \overset{\Psi'} {\longrightarrow} \quad  \mbox{BNCs for }\hat{\mu}'.
\end{equation*}
The Jacobian matrix of $\Psi$ and its inverse at $x_{3}=0$ are given by
\begin{equation}\label{DPSieqs}
D\Psi =\left( \begin{matrix}
1&0& -\frac{\hat{\mu}^{13}}{\hat{\mu}^{33}}\\
0&1& -\frac{\hat{\mu}^{23}}{\hat{\mu}^{33}}\\
0&0& \frac{1}{\sqrt{\hat{\mu}^{33}}}
\end{matrix}\right), \quad 
D\Psi^{-1}=\left( \begin{matrix}
1&0& \frac{\hat{\mu}^{13}}{\sqrt{\hat{\mu}^{33}}}\\
0&1& \frac{\hat{\mu}^{23}}{\sqrt{\hat{\mu}^{33}}}\\
0&0& {\sqrt{\hat{\mu}^{33}}}
\end{matrix}\right),
\end{equation}
where the entries of $\hat{\mu}$ are expressed in boundary normal coordinates for $\hat{\varepsilon}$. Similarly, the Jacobian matrix of $\Psi'$ and its inverse at $x_{3}=0$ are obtained by \eqref{DPSieqs} after substituting the entries of $\hat{\mu}$ by the entries of $\hat{\mu}'$.
Using \eqref{DPSieqs} we obtain the following representations for $\hat{\varepsilon}$ and its determinant in boundary normal coordinates for $\hat{\mu}$ at $x_{3}=0$
\begin{equation}\label{epsilonbncsmu}
\hat{\varepsilon}= D\Psi \left(\begin{matrix}
\tilde{\varepsilon}&0\\
0^{T}&1
\end{matrix}
\right) (D\Psi)^{T}= \left(\begin{matrix}
\tilde{\varepsilon}+ \frac{\hat{\mu}^{3\tilde{i}}\hat{\mu}^{3\tilde{j}}}{(\hat{\mu}^{33})^{2}}& - \frac{\hat{\mu}^{\tilde{i}3}}{(\hat{\mu}^{33})^{3/2}}\\
- \frac{\hat{\mu}^{\tilde{j}3}}{(\hat{\mu}^{33})^{3/2}}& \frac{1}{\hat{\mu}^{33}}
\end{matrix}
\right), \quad \lvert \hat{\varepsilon}\rvert=\frac{\lvert\tilde{\varepsilon}\rvert}{\hat{\mu}^{33}}.
\end{equation}
Similarly, the metric $\hat{\mu}$ and its determinant in boundary normal coordinates for $\hat{\varepsilon}$ at $x_{3}=0$ are provided by
\begin{equation}\label{mubncsepsilon}
\hat{\mu}= D\Psi ^{-1}\left(\begin{matrix}
\tilde{\mu}&0\\
0^{T}& 1
\end{matrix}\right) (D\Psi ^{-1})^{T}=\left( \begin{matrix}
\tilde{\mu}+ \frac{\hat{\mu}^{3\tilde{i}}\hat{\mu}^{3\tilde{j}}}{\hat{\mu}^{33}}& \hat{\mu}^{\tilde{i}3}\\
\hat{\mu}^{3\tilde{j}}&\hat{\mu}^{33}
\end{matrix}\right), \quad \lvert\hat{\mu}\rvert=\hat{\mu}^{33}\lvert\tilde{\mu}\rvert.
\end{equation}
The expressions of $\hat{\varepsilon}'$ in boundary normal coordinates for $\hat{\mu}'$ and $\hat{\mu}'$ in boundary normal coordinates for $\hat{\varepsilon}'$ at $x_{3}=0$ follow from equations \eqref{epsilonbncsmu} and \eqref{mubncsepsilon}, respectively, after substituting the entries of $\hat{\mu}$ by the entries of $\hat{\mu}'$.
Now let us introduce the notation for the covectors  $\xi_{\hat{\mu}}$, $\chi_{\hat{\mu}}$, $\xi_{\hat{\varepsilon}}$ and $\chi_{\hat{\varepsilon}}$ in boundary normal coordinates for $\hat{\mu}$ and $\hat{\varepsilon}$ as shown in Table \ref{covectorsinbncs} and then express their normal components in each coordinate system.
\begin{table}[h]
	\begin{center}
	\begin{tabular}{|c| c| c| c|} 
		\hline
	Covectors&	BNCs for $\hat{\mu}$&   BNCs for $\hat{\varepsilon}$ \\ [0.5ex] 
		\hline
	$\xi_{\hat{\mu}}$&	$\xi_{\hat{\mu},\tilde{\mu}}$ &  $\xi_{\hat{\mu},\tilde{\varepsilon}}$   \\ 
		\hline
	$\chi_{\hat{\mu}}$&	$\chi_{\hat{\mu},\tilde{\mu}}$ & $\chi_{\hat{\mu},\tilde{\varepsilon}}$   \\
		\hline
	$\xi_{\hat{\varepsilon}}$&	$\xi_{\hat{\varepsilon},\tilde{\mu}}$& $\xi_{\hat{\varepsilon},\tilde{\varepsilon}}$  \\
		\hline
	$\chi_{\hat{\varepsilon}}$&	$\chi_{\hat{\varepsilon},\tilde{\mu}}$	 &  $\chi_{\hat{\varepsilon},\tilde{\varepsilon}}$ \\
		\hline
		\end{tabular}
	\caption{\label{covectorsinbncs}Notation for covectors in the different coordinate systems.}
	\end{center}
\end{table}
\begin{lemma}\label{xilemma}
	The normal components of the covectors $\xi_{\hat{\varepsilon},\tilde{\varepsilon}}$, $\xi_{\hat{\mu},\tilde{\mu}}$, $\xi_{\hat{\mu},\tilde{\varepsilon}}$ and $\xi_{\hat{\varepsilon},\tilde{\mu}}$ are provided by
	\begin{equation*}
{\xi_{\hat{\varepsilon},\tilde{\varepsilon}}}_{3}= i \lvert\tilde{\xi}\rvert_{\tilde{\varepsilon}}, \quad 
{\xi_{\hat{\mu},\tilde{\mu}}}_{3}= i \lvert\tilde{\xi}\rvert_{\tilde{\mu}},
\end{equation*}
	\begin{equation*}
{\xi_{\hat{\mu},\tilde{\varepsilon}}}_{3}= - \frac{\langle \nu,\tilde{\xi} \rangle_{\hat{\mu}}}{\hat{\mu}^{33}} + i \frac{\lvert\tilde{\xi}\rvert_{\tilde{\mu}}}{\sqrt{\hat{\mu}^{33}}}, \quad {\xi_{\hat{\varepsilon},\tilde{\mu}}}_{3}= \frac{\langle \nu,\tilde{\xi}\rangle_{\hat{\mu}}}{\sqrt{\hat{\mu}^{33}}}+ i \lvert\tilde{\xi}\rvert_{\tilde{\varepsilon}}\sqrt{\hat{\mu}^{33}}.
\end{equation*}
\end{lemma}
\begin{proof}
	To prove the above equations it suffices to use equations \eqref{xihatmu3} and \eqref{xihatepsilon3} and the representations of each metric in the different coordinate systems.
\end{proof}
\begin{lemma}\label{chilemma}
	The normal components of the covectors ${\chi_{\hat{\mu},\tilde{\varepsilon}}}_{3}$, ${\chi_{\hat{\mu}',\tilde{\varepsilon}}}_{3}$, ${\chi_{\hat{\varepsilon},\tilde{\mu}}}_{3}$ and  ${\chi_{\hat{\varepsilon}',\tilde{\mu}}}_{3}$ are expressed as
\begin{equation*}
{\chi_{\hat{\mu},\tilde{\varepsilon}}}_{3}= \frac{\sigma_{3\tilde{j}\tilde{i}}\tilde{\varepsilon}^{\tilde{i}\tilde{d}}\xi_{\tilde{d}} \tilde{\mu}^{\tilde{j}\tilde{b}}\xi_{\tilde{b}}}{\sqrt{\lvert\tilde{\mu}\rvert}\sqrt{\hat{\mu}^{33}}}+ i \lvert\tilde{\xi}\rvert_{\tilde{\mu}} \frac{\sigma_{3\tilde{j}\tilde{i}}\tilde{\varepsilon}^{\tilde{i}\tilde{b}}\xi_{\tilde{b}}\hat{\mu}^{\tilde{j}3}}{\sqrt{\lvert\tilde{\mu}\rvert}\hat{\mu}^{33}},
\end{equation*}	
	\begin{equation*}
{\chi_{\hat{\varepsilon},\tilde{\mu}}}_{3}= \frac{\sigma_{3\tilde{j}\tilde{i}}\tilde{\varepsilon}^{\tilde{b}\tilde{j}}\xi_{\tilde{b}}\tilde{\mu}^{\tilde{i}\tilde{d}}\xi_{\tilde{d}}\sqrt{\hat{\mu}^{33}}}{\sqrt{\lvert\tilde{{\varepsilon}}\rvert}}- i\lvert\tilde{\xi}\rvert_{\tilde{\varepsilon}} \frac{\sigma_{3\tilde{j}\tilde{i}}\hat{\mu}^{3\tilde{j}}\tilde{\mu}^{\tilde{i}\tilde{d}}\xi_{\tilde{d}}}{\sqrt{\lvert\tilde{\varepsilon}\rvert}\sqrt{\hat{\mu}^{33}}}.
\end{equation*}
\end{lemma}
\begin{proof}
	We use the following expressions of the covectors $\chi_{\hat{\mu}}$ and $\chi_{\hat{\varepsilon}}$ in general coordinates
		\begin{equation*}
	{\chi_{\hat{\mu}}}_{k}=\frac{\sigma_{kji}\hat{\mu}^{bj}{\xi_{\hat{\mu}}}_{b}\hat{\varepsilon}^{id}{\xi_{\hat{\mu}}}_{d}}{\sqrt{\lvert\hat{\mu}\rvert}},
	\end{equation*}
	\begin{equation*}
	{\chi_{\hat{\varepsilon}}}_{k}=\frac{\sigma_{kji}\hat{\varepsilon}^{bj}{\xi_{\hat{\varepsilon}}}_{b}\hat{\mu}^{id}{\xi_{\hat{\varepsilon}}}_{d}}{\sqrt{\lvert\hat{\varepsilon}\rvert}},
	\end{equation*}
	which are derived using the identity \eqref{altdetinvid} for the metrics $\hat{\mu}$ and $\hat{\varepsilon}$, respectively.
	 The normal components of these covectors in each coordinate system are calculated using the representations of the metrics in each coordinate system and Lemma \ref{xilemma}.  
\end{proof}
The normal components of the $'-$ covectors in the different coordinate systems denoted by ${\xi_{\hat{\varepsilon}',\tilde{\varepsilon}}}_{3}$, ${\xi_{\hat{\mu}',\tilde{\mu}}}_{3}$, $\xi_{\hat{\mu}',\tilde{\varepsilon}}$, $\xi_{\hat{\varepsilon}',\tilde{\mu}}$, ${\chi_{\hat{\varepsilon}',\tilde{\mu}}}_{3}$ and ${\chi_{\hat{\mu}',\tilde{\varepsilon}}}_{3}$ are obtained from the expressions derived in Lemma \ref{xilemma} and \ref{chilemma} after substituting the normal components of $\hat{\mu}$ by the normal components of $\hat{\mu}'$.
\begin{lemma}\label{chimumu'lemma}
	Let $(\hat{\varepsilon},\hat{\mu})$ $(\hat{\varepsilon}',\hat{\mu}')$ be 2 different sets of metrics and $\Lambda_{\hat{\varepsilon}}$, $\Lambda_{\hat{\varepsilon}'}$ the associated boundary mappings. 
	If $\Lambda_{\hat{\varepsilon}}=\Lambda_{\hat{\varepsilon}'}$, then the following identities hold
	\begin{equation}\label{chimumu'id}
	\left({\xi_{\hat{\mu}',\tilde{\varepsilon}}}_{3}+ {\xi_{\hat{\varepsilon},\tilde{\varepsilon}}}_{3}\right){\chi_{\hat{\mu},\tilde{\varepsilon}}}_{3}=\left({\xi_{\hat{\mu},\tilde{\varepsilon}}}_{3}+ {\xi_{\hat{\varepsilon},\tilde{\varepsilon}}}_{3}\right){\chi_{\hat{\mu}',\tilde{\varepsilon}}}_{3}, \quad \mbox{at} \ x_{3}=0,
	\end{equation}
		\begin{equation}\label{chiepsilonepsilon'id}
	\left({\xi_{\hat{\varepsilon}',\tilde{\mu}}}_{3}+ {\xi_{\hat{\mu},\tilde{\mu}}}_{3}\right){\chi_{\hat{\varepsilon},\tilde{\mu}}}_{3}=\left({\xi_{\hat{\varepsilon},\tilde{\mu}}}_{3}+ {\xi_{\hat{\mu},\tilde{\mu}}}_{3}\right){\chi_{\hat{\varepsilon}',\tilde{\mu}}}_{3}, \quad \mbox{at} \ x_{3}=0.
	\end{equation}
\end{lemma}
\begin{proof}
	We will show the derivation of the identity \eqref{chimumu'id} as \eqref{chiepsilonepsilon'id} is proven in a similar way. Let us fix boundary normal coordinates for $\hat{\varepsilon}$. We aim to write $\tilde{\chi}_{\hat{\mu}}$ as
	\begin{equation}\label{chimuxibasis}
	\tilde{\chi}_{\hat{\mu}}= a \tilde{\xi}+ b \tilde{\xi}_{\tilde{\varepsilon} \perp}
	\end{equation}
for some constants $a$, $b$ to be determined and $\tilde{\xi}_{\tilde{\varepsilon} \perp}:= \ast_{\tilde{\varepsilon}}\tilde{\xi}$ satisfying
	\begin{equation}
	\langle \tilde{\xi}, \tilde{\xi}_{\tilde{\varepsilon} \perp} \rangle_{\tilde{\varepsilon}}=0, \quad \tilde{\xi} \wedge \tilde{\xi}_{\tilde{\varepsilon} \perp}= \lvert \tilde{\xi}\rvert_{\tilde{\varepsilon}}^{2} \ dv_{\tilde{\varepsilon}},
	\end{equation}
	where $dv_{\tilde{\varepsilon}}$ denotes the volume form at the boundary with respect to the $\tilde{\varepsilon}$ metric. Considering the inner product with respect to the $\tilde{\varepsilon}$ metric of $\tilde{\xi}$ with \eqref{chimuxibasis}, we obtain
	\begin{equation*}
	a= - \frac{{\xi_{\hat{\mu},\tilde{\varepsilon}}}_{3}{\chi_{\hat{\mu},\tilde{\varepsilon}}}_{3}}{\lvert \tilde{\xi} \rvert_{\tilde{\varepsilon}}^{2}},
	\end{equation*}
where we have used the fact that $\langle \chi_{\hat{\mu},\tilde{\varepsilon}},\xi_{\hat{\mu}}\rangle_{\hat{\varepsilon}}=0$. The wedge product of $\tilde{\xi}$ with \eqref{chimuxibasis} results to
		\begin{equation*}
	b= - i\frac{\sqrt{\lvert \tilde{\varepsilon}\rvert}\lvert \xi_{\hat{\mu}}\rvert_{\hat{\varepsilon}}^{2}\lvert \tilde{\xi}\rvert_{\tilde{\mu}}}{\sqrt{\lvert \tilde{\mu}\rvert}\lvert \tilde{\xi}\rvert_{\tilde{\varepsilon}}^{2}}.
	\end{equation*}
	Now let us write the principal symbol of the normal component of $E$ as
	\begin{equation*}
	E_{3}^{(0)}= a_{\hat{\mu}}{\chi_{\hat{\mu},\tilde{\varepsilon}}}_{3}+ b_{\hat{\mu}} {\xi_{\hat{\varepsilon},\tilde{\varepsilon}}}_{3}.
	\end{equation*}
	Making use of equations \eqref{abepsilonconst} as well as equation \eqref{chimuxibasis} we arrive at
	\begin{equation*}
E_{3}^{(0)}(G)= \frac{{\chi_{\hat{\mu},\tilde{\varepsilon}}}_{3} }{\lvert \tilde{\xi}\rvert_{\tilde{\mu}}\lvert \tilde{\xi}\rvert_{\tilde{\varepsilon}}({\xi_{\hat{\mu},\tilde{\varepsilon}}}_{3}+ {\xi_{\hat{\varepsilon},\tilde{\varepsilon}}}_{3}) } \ast_{\iota^*\hat{\mu}}\left (\tilde{\xi}\wedge G\right )  + \frac{i \sqrt{\lvert \tilde{\varepsilon}\rvert}}{\sqrt{\lvert \tilde{\mu}\rvert}\lvert\tilde{\xi}\rvert_{\tilde{\varepsilon}}} \ast_{\iota^*\hat{\mu}} \left (\tilde{\xi}_{\tilde{\varepsilon} \perp} \wedge G \right ).
\end{equation*}
Setting this equal to $E_{3}'^{(0)}(G)$ with $G = \tilde{\xi}_{\tilde{\varepsilon}\perp}$ and using Theorem \ref{fieldsequalattheboundary} completes the proof.
\end{proof}

\subsection{Proof of Theorem \ref{normalmuthm}} \label{normalmusec}

Using Lemma \ref{xilemma} and \ref{chilemma}, the imaginary parts of \eqref{chimumu'id} and \eqref{chiepsilonepsilon'id}  are provided by, respectively,
	\begin{equation}\label{impartsyst}
	\begin{split}
\sigma_{3\tilde{j}\tilde{i}}\tilde{\varepsilon}^{\tilde{i}\tilde{d}}\xi_{\tilde{d}}\left(  \langle \nu,\tilde{\xi}\rangle_{\hat{\mu}}\hat{\mu}'^{3\tilde{j}}-\langle \nu,\tilde{\xi}\rangle_{\hat{\mu}'}\hat{\mu}^{3\tilde{j}}
\right)&=\frac{\lvert\tilde{\xi}\rvert_{\tilde{\varepsilon}}}{\lvert\tilde{\xi}\rvert_{\tilde{\mu}}} \sigma_{3\tilde{i}\tilde{j}}\tilde{\varepsilon}^{\tilde{i}\tilde{d}}\xi_{\tilde{d}}\tilde{\mu}^{\tilde{j}\tilde{b}}\xi_{\tilde{b}}\left(\hat{\mu}'^{33}\sqrt{\hat{\mu}^{33}} -\hat{\mu}^{33}\sqrt{\hat{\mu}'^{33}} \right),\\
\sigma_{3\tilde{j}\tilde{i}}\tilde{\mu}^{\tilde{i}\tilde{d}}\xi_{\tilde{d}}\left(  \langle \nu,\tilde{\xi}\rangle_{\hat{\mu}}\hat{\mu}'^{3\tilde{j}}-\langle \nu,\tilde{\xi}\rangle_{\hat{\mu}'}\hat{\mu}^{3\tilde{j}}
\right)&=\frac{\lvert\tilde{\xi}\rvert_{\tilde{\mu}}}{\lvert\tilde{\xi}\rvert_{\tilde{\varepsilon}}} \sigma_{3\tilde{i}\tilde{j}}\tilde{\varepsilon}^{\tilde{i}\tilde{d}}\xi_{\tilde{d}}\tilde{\mu}^{\tilde{j}\tilde{b}}\xi_{\tilde{b}}\left(\hat{\mu}'^{33}\sqrt{\hat{\mu}^{33}} -\hat{\mu}^{33}\sqrt{\hat{\mu}'^{33}} \right).
\end{split}
\end{equation}
Combining the above displayed formulae we arrive at
\begin{equation}\label{parallelvecteq}
\sigma_{3\tilde{j}\tilde{i}}\left(\langle \nu,\tilde{\xi}\rangle_{\hat{\mu}}\hat{\mu}'^{3\tilde{j}}-  \langle \nu,\tilde{\xi}\rangle_{\hat{\mu}'}\hat{\mu}^{3\tilde{j}}\right) \left( \lvert\tilde{\xi}\rvert_{\tilde{\varepsilon}}^{2} \tilde{\mu}^{\tilde{i}\tilde{d}}\xi_{\tilde{d}}-\lvert\tilde{\xi}\rvert_{\tilde{\mu}}^{2}\tilde{\varepsilon}^{\tilde{i}\tilde{d}}\xi_{\tilde{d}}
\right)=0.
\end{equation}
Now we are going distinguish the following two different cases. 
\begin{enumerate}
	\item There exists $a>0$ such that \label{case1}
	\begin{equation}\label{mutildeepsilontildemultiples}
	\tilde{\mu}=a \tilde{\varepsilon}.
	\end{equation} 
	\item The metrics $\tilde{\mu}$ and $\tilde{\varepsilon}$ are not multiples.\label{case2}
\end{enumerate}
Starting with case \ref{case1}, we have that $\sigma_{3\tilde{i}\tilde{j}}\tilde{\varepsilon}^{\tilde{i}\tilde{d}}\xi_{\tilde{d}}\tilde{\mu}^{\tilde{j}\tilde{b}}\xi_{\tilde{b}}=0$ and the first equation in \eqref{impartsyst} reduces to
\begin{equation*}
\sigma_{3\tilde{j}\tilde{i}}\tilde{\varepsilon}^{\tilde{i}\tilde{d}}\xi_{\tilde{d}}\left(  \langle \nu,\tilde{\xi}\rangle_{\hat{\mu}}\hat{\mu}'^{3\tilde{j}}-\langle \nu,\tilde{\xi}\rangle_{\hat{\mu}'}\hat{\mu}^{3\tilde{j}}
\right)=0.
\end{equation*}
This implies that if $\left(  \langle \nu,\tilde{\xi}\rangle_{\hat{\mu}}\hat{\mu}'^{3\tilde{j}}-\langle \nu,\tilde{\xi}\rangle_{\hat{\mu}'}\hat{\mu}^{3\tilde{j}}
\right) \neq 0$, there exists $b\in \mathbb{R}$ such that
\begin{equation*}
\tilde{\varepsilon}^{\tilde{j}\tilde{d}}\xi_{\tilde{d}}= b  \left( \langle \nu,\tilde{\xi}\rangle_{\hat{\mu}} \hat{\mu}'^{3\tilde{j}}-\langle \nu,\tilde{\xi}\rangle_{\hat{\mu}'}\hat{\mu}^{3\tilde{j}}
\right),
\end{equation*}
which cannot hold, since multiplying the above by $\xi_{\tilde{j}}$ leads to $\lvert\tilde{\xi}\rvert_{\varepsilon}=0$. Therefore, $\left( \langle \nu,\tilde{\xi}\rangle_{\hat{\mu}} \hat{\mu}'^{3\tilde{j}}-\langle \nu,\tilde{\xi}\rangle_{\hat{\mu}'}\hat{\mu}^{3\tilde{j}}
\right)=0$ or equivalently
\begin{equation*}
\hat{\mu}'^{3\tilde{j}}=c\hat{\mu}^{3\tilde{j}}, \quad \mbox{at} \ x_{3}=0.
\end{equation*}
for $c\in \mathbb{R}$.

It remains to work with case \ref{case2}, where $\tilde{\mu}$ and $\tilde{\varepsilon}$ are assumed not to be multiples. Under this assumption, equation \eqref{parallelvecteq} yields, at least for all $\tilde{\xi}$ in an open set, to
\begin{equation*}
\langle \nu,\tilde{\xi} \rangle_{\hat{\mu}}\hat{\mu}'^{3\tilde{j}}- \langle \nu,\tilde{\xi} \rangle_{\hat{\mu}'} \hat{\mu}^{3\tilde{j}}= \beta \left(  \tilde{\mu}^{\tilde{j}\tilde{a}}\xi_{\tilde{a}}\lvert\tilde{\xi}\rvert_{\tilde{\varepsilon}}^{2}-\tilde{\varepsilon}^{\tilde{j}\tilde{a}}\xi_{\tilde{a}}\lvert\tilde{\xi}\rvert_{\tilde{\mu}}^{2}
\right),
\end{equation*}
where $\beta \in \mathbb{R}$. Combining the above with the equations in \eqref{impartsyst} we get
\begin{equation}\label{parallelvecteq2}
\langle \nu,\tilde{\xi}\rangle_{\hat{\mu}}\hat{\mu}'^{3\tilde{j}}- \langle \nu,\tilde{\xi}\rangle_{\hat{\mu}'} \hat{\mu}^{3\tilde{j}}=\left( \hat{\mu}^{33}\sqrt{\hat{\mu}'^{33}}-\hat{\mu}'^{33}\sqrt{\hat{\mu}^{33}}\right) \left(  \tilde{\mu}^{\tilde{j}\tilde{a}}\xi_{\tilde{a}}\frac{\lvert\tilde{\xi}\rvert_{\tilde{\varepsilon}}}{\lvert\tilde{\xi}\rvert_{\tilde{\mu}}}-\tilde{\varepsilon}^{\tilde{j}\tilde{a}}\xi_{\tilde{a}}\frac{\lvert\tilde{\xi}\rvert_{\tilde{\mu}}}{\lvert\tilde{\xi}\rvert_{\tilde{\varepsilon}}}
\right).
\end{equation}
Differentiating \eqref{parallelvecteq2} with respect to $\tilde{\xi}$ results to
\begin{equation*}
\begin{split}
\hat{\mu}'^{3\tilde{j}}\hat{\mu}^{3\tilde{a}}-  \hat{\mu}^{3\tilde{j}}\hat{\mu}'^{3\tilde{a}}&=\left( \hat{\mu}^{33}\sqrt{\hat{\mu}'^{33}}-\hat{\mu}'^{33}\sqrt{\hat{\mu}^{33}}\right)\left(\tilde{\mu}^{\tilde{j}\tilde{a}}\frac{\lvert\tilde{\xi}\rvert_{\tilde{\varepsilon}}}{\lvert\tilde{\xi}\rvert_{\tilde{\mu}}}+ \frac{\tilde{\mu}^{\tilde{j}\tilde{d}}\xi_{\tilde{d}}}{\lvert\tilde{\xi}\rvert_{\tilde{\mu}}}\frac{\tilde{\varepsilon}^{\tilde{a}\tilde{b}}\xi_{\tilde{b}}}{\lvert\tilde{\xi}\rvert_{\tilde{\varepsilon}}}\right.\\
&\left. - \tilde{\mu}^{\tilde{j}\tilde{d}}\xi_{\tilde{d}}\lvert\tilde{\xi}\rvert_{\tilde{\varepsilon}} \frac{\tilde{\mu}^{\tilde{b}\tilde{a}}\xi_{\tilde{b}}}{\lvert\tilde{\xi}\rvert_{\tilde{\mu}}^{3}}-\tilde{\varepsilon}^{\tilde{j}\tilde{a}}\frac{\lvert\tilde{\xi}\rvert_{\tilde{\mu}}}{\lvert\tilde{\xi}\rvert_{\tilde{\varepsilon}}}- \frac{\tilde{\varepsilon}^{\tilde{j}\tilde{d}}\xi_{\tilde{d}}}{\lvert\tilde{\xi}\rvert_{\tilde{\varepsilon}}}\frac{\tilde{\mu}^{\tilde{a}\tilde{b}}\xi_{\tilde{b}}}{\lvert\tilde{\xi}\rvert_{\tilde{\mu}}}+ \tilde{\varepsilon}^{\tilde{j}\tilde{d}}\xi_{\tilde{d}}\lvert\tilde{\xi}\rvert_{\tilde{\mu}} \frac{\tilde{\varepsilon}^{\tilde{b}\tilde{a}}\xi_{\tilde{b}}}{\lvert\tilde{\xi}\rvert_{\tilde{\varepsilon}}^{3}}\right).
\end{split}
\end{equation*}
Let $\eta \in T^*_x(\partial M)$ be linearly independent from $\tilde{\xi}$. Multiplying both sides of the above displayed equation by $\eta$, we get
\begin{equation}\label{xietaeq}
\left( \hat{\mu}^{33}\sqrt{\hat{\mu}'^{33}}-\hat{\mu}'^{33}\sqrt{\hat{\mu}^{33}}\right)\left[\frac{\lvert\tilde{\xi}\rvert_{\tilde{\varepsilon}}}{\lvert\tilde{\xi}\rvert_{\tilde{\mu}}} \left( \lvert\eta\rvert_{\tilde{\mu}}^{2}-
\frac{\langle \tilde{\xi},\eta \rangle _{\tilde{\mu}}^{2}}{\lvert\tilde{\xi}\rvert_{\tilde{\mu}}^{2}}\right)-\frac{\lvert\tilde{\xi}\rvert_{\tilde{\mu}}}{\lvert\tilde{\xi}\rvert_{\tilde{\varepsilon}}} \left( \lvert\eta\rvert_{\tilde{\varepsilon}}^{2}-
\frac{\langle \tilde{\xi},\eta \rangle _{\tilde{\varepsilon}}^{2}}{\lvert\tilde{\xi}\rvert_{\tilde{\varepsilon}}^{2}}\right)\right]=0.
\end{equation}
Assuming that the second factor in \eqref{xietaeq} vanishes and fixing $\eta$ such that
\begin{equation*}
\eta_{1}=-\xi_{2}, \quad \eta_{2}=\xi_{1}
\end{equation*}
we obtain
\begin{equation*}
\frac{\lvert\tilde{\xi}\rvert_{\tilde{\varepsilon}}^{4}}{\lvert\tilde{\xi}\rvert_{\tilde{\mu}}^{4}}= \frac{\lvert\tilde{\varepsilon}\rvert}{\lvert\tilde{\mu}\rvert}.
\end{equation*}
The above holds for $\tilde{\xi}$ in an open set, and we observe that the right hand-side does not depend on $\tilde{\xi}$. Therefore $\tilde{\varepsilon}$ should be a multiple of $\tilde{\mu}$, which contradicts to our initial assumption. Hence, the second factor in \eqref{xietaeq} cannot be zero
which leads to
\begin{equation*}
\hat{\mu}^{33}=\hat{\mu}'^{33}, \quad \mbox{at} \ x_{3}=0.
\end{equation*} 
Taking this into account and revisiting equation \eqref{parallelvecteq2} we have, assuming without loss of generality that $\hat{\mu}^{\tilde{j}3} \neq 0$,
\begin{equation*}
\hat{\mu}'^{\tilde{j}3}=c\hat{\mu}^{\tilde{j}3}, \quad \mbox{at} \ x_{3}=0,
\end{equation*}
for $c\in\mathbb{R}$. Next, the real part of \eqref{chimumu'id}, by virtue of the relationships between the normal components of $\hat{\mu}$ and $\hat{\mu}'$, becomes
	\begin{equation}\label{realpartmu2}
\begin{split}
\left(1-c \right)\langle \tilde{\xi},\nu \rangle _{\hat{\mu}}\sigma_{3\tilde{j}\tilde{i}}\tilde{\varepsilon}^{\tilde{i}\tilde{d}}\xi_{\tilde{d}} \tilde{\mu}^{\tilde{i}\tilde{b}}\xi_{\tilde{b}}
+\lvert\tilde{\xi}\rvert_{\tilde{\mu}}^{2}\left( c-1\right) \sigma_{3\tilde{j}\tilde{i}}\tilde{\varepsilon}^{\tilde{i}\tilde{d}}\xi_{\tilde{d}}\hat{\mu}^{3\tilde{j}}\\=\lvert\tilde{\xi}\rvert_{\tilde{\varepsilon}} \lvert\tilde{\xi}\rvert_{\tilde{\mu}}\sqrt{\hat{\mu}^{33}}\left( 1-c \right)\sigma_{3\tilde{j}\tilde{i}}\tilde{\varepsilon}^{\tilde{i}\tilde{d}}\xi_{\tilde{d}}\hat{\mu}^{3\tilde{j}},
\end{split}
\end{equation}
which choosing $\tilde{\xi}\neq 0$ such that
\begin{equation}\label{tildexieq}
\langle \nu,\tilde{\xi}\rangle_{\hat{\mu}}=0,
\end{equation}
simplifies to
\begin{equation*}
\left( c-1\right) \left( \lvert\tilde{\xi}\rvert_{\tilde{\mu}}^{2}+\lvert\tilde{\xi}\rvert_{\tilde{\varepsilon}} \lvert\tilde{\xi}\rvert_{\tilde{\mu}}\sqrt{\hat{\mu}^{33}}\right)\sigma_{3\tilde{j}\tilde{i}}\tilde{\varepsilon}^{\tilde{i}\tilde{d}}\xi_{\tilde{d}}\hat{\mu}^{3\tilde{j}}=0.
\end{equation*}
Since the vectors $\tilde{\varepsilon}^{\tilde{j}\tilde{d}}\xi_{\tilde{d}}$ and $\hat{\mu}^{3{\tilde{j}}}$ are not multiples of each other, we have
\begin{equation*}
\left( c-1\right) \left( \lvert\tilde{\xi}\rvert_{\tilde{\mu}}^{2}+\lvert\tilde{\xi}\rvert_{\tilde{\varepsilon}} \lvert\tilde{\xi}\rvert_{\tilde{\mu}}\sqrt{\hat{\mu}^{33}}\right)=0.
\end{equation*}
Since the second term on the left is positive, we deduce that $c=1$ which implies
\begin{equation*}
\hat{\mu}^{3\tilde{j}}=	\hat{\mu}'^{3\tilde{j}}, \quad \mbox{at} \ x_{3}=0.
\end{equation*}
This completes the proof of Theorem \ref{normalmuthm}.

\hfill\openbox

In the following lemma we show that when the different pairs of metrics agree at $x_{3}=0$ in boundary normal coordinates for $\hat{\varepsilon}$/$\hat{\varepsilon}'$ then the same holds for the electromagnetic fields $E$, $H$, $E'$, $H'$. This will be used in Section \ref{normalderivativessec}.
	\begin{lemma}\label{prinsymbEE'HH'equallemma}
		Let $(\hat{\varepsilon},\hat{\mu})$, $(\hat{\varepsilon}',\hat{\mu}')$ be two different sets of metrics and $E$, $H$, $E'$, $H'$ be the corresponding electric and magnetic fields. If $\hat{\varepsilon}=\hat{\varepsilon}'$ and $\hat{\mu}=\hat{\mu}'$ at $x_{3}=0$, in boundary normal coordinates for $\hat{\varepsilon}$/$\hat{\varepsilon}'$, then
		\begin{equation}
		E=E', \quad \mbox{at} \ x_{3}=0
		\end{equation}
		and
		\begin{equation}
		H=H', \quad \mbox{at} \ x_{3}=0.
		\end{equation}
	\end{lemma}
\begin{proof}
	To prove the above we only need to consider the magnetic fields $H$ and $H'$ as the result for $E$, $E'$ is already proven in Lemma \ref{fieldsequalattheboundary}. Considering the change of coordinates \eqref{DPSieqs} and the fact that $H=H'$ in boundary normal coordinates for $\hat{\mu}$/$\hat{\mu}'$ completes the proof.
\end{proof}


	\section{Normal Derivatives of Electromagnetic Parameters}\label{normalderivativessec}
	

In this section, we apply the methods of \cite{Joshi} and \cite{Lionheart} to show inductively that when the metrics $(\hat{\varepsilon},\hat{\mu})$,$(\hat{\varepsilon}',\hat{\mu}')$ are equal at $x_{3}=0$, then, for any $\kappa>0$, $\partial_{x_{3}}^{\kappa}\hat{\varepsilon}=\partial_{x_{3}}^{\kappa}\hat{\varepsilon}'$, $\partial_{x_{3}}^{\kappa}\hat{\mu}=\partial_{x_{3}}^{\kappa}\hat{\mu}'$ at $x_{3}=0$ which proves Theorem \ref{normalderivativesthm}.

Let us consider the pseudodifferential operators $B(x,D_{\tilde{x}})$ and $B'(x,D_{\tilde{x}})$ associated with the metrics $(\hat{\varepsilon}, \hat{\mu})$ and $(\hat{\varepsilon}', \hat{\mu}')$ respectively and define $Y=B'-B$ $\in$ $\Psi DO^{(1)}$. Making use of the Riccati equation \eqref{RiccatiB} and the corresponding Riccati equation for the $'-$ parameters we have that $Y$ satisfies
	\begin{equation}\label{YRiccati}
\begin{split}
&T_{H'}D_{x^{3}}Y+ ( A_{H'}+ G_{H'})Y+ T_{H'}Y^{2}+ T_{H'}BY +T_{H'}YB=\\&\hskip0.5cm [(A_{H}-A_{H'})+(G_{H}-G_{H'})+(T_{H}-T_{H'})B]B + (R_{H}-R_{H'})\\&\hskip0.5cm+ (Q_{H}-Q_{H'}) +(F_{H}-F_{H'})+(T_{H}-T_{H'})D_{x^{3}}B.
\end{split}
\end{equation}
Similarly, $J$ $\in$ $\Psi DO^{(1)}$ realised by $J=C'-C$, satisfies
	\begin{equation}\label{JRiccati}
\begin{split}
&T_{E'}D_{x^{3}}J+ ( A_{E'}+ G_{E'})J+ T_{E'}J^{2}+ T_{E'}CJ +T_{E'}JC=\\&\hskip0.5cm [(A_{E}-A_{E'})+(G_{E}-G_{E'})+(T_{E}-T_{E'})C]C + (R_{E}-R_{E'})\\&\hskip0.5cm + (Q_{E}-Q_{E'}) +(F_{E}-F_{E'})+(T_{E}-T_{E'})D_{x^{3}}C.
\end{split}
\end{equation}
Let us now define a class of pseudodifferential operators that depends smoothly on the normal distance from the boundary with respect to the $\hat{\varepsilon}$ metric.
\begin{definition}[$\Psi DO^{(m,p)}$]\label{familyofPsiDosdef}
	We say $P$ $\in\Psi DO^{(m,p)}(\partial M, \mathbb{R}_{+})$ if it is a family of pseudodifferential operators of order $m$ on $\partial M$, varying smoothly up to $x_{3}=0$ such that
	\begin{equation*}
	P=\sum_{j=0}^{p}x_{3}^{j}P^{(m-p+j)},
	\end{equation*}
	with $P^{(j)}$ $\in$ $\Psi DO^{(m-p+j)}(\partial M)$.
\end{definition}
The calculus of $\Psi DO$s as well as the definition immediately implies the following results which are collected here for later use.
\begin{lemma}[Calculus for $\Psi DO^{(m,p)}$] \label{PsiDOcalclem}
Suppose that $P \in \Psi DO^{(m,p)}(\partial M, \mathbb{R}_{+})$ and $P' \in \Psi DO^{(m',p')}(\partial M, \mathbb{R}_{+})$. Then $P P' \in \Psi DO^{(m+m',p+p')}(\partial M, \mathbb{R}_{+})$. Also, $P \in \Psi DO^{(m+1,p+1)}(\partial M, \mathbb{R}_{+})$ and $x_3 P \in \Psi DO^{(m,p+1)}(\partial M, \mathbb{R}_{+})$
\end{lemma}

\begin{definition}[Principal symbol in $\Psi DO^{(m,p)}$] \label{princsymbdef}
	For $P \in \Psi DO^{(m,p)}(\partial M, \mathbb{R}_{+})$, we define the vector of principal symbols of $P=\sum_{j=0}^{p}x_{3}^{j}P^{(m-p+j)}$ at $x_{3}=0$, by taking 
	\begin{equation*}
	\sigma_{p,l}(P) = \sigma_p\left (\frac{1}{l!}\frac{d^{l}}{dx_{3}^{l}}P\Big\vert_{x_{3}=0}\right ),
	\end{equation*}
	where $0\leq l \leq p$.
	More particularly, we have that
	\begin{equation*}
	\sigma_p(P) = \left (\sigma_p\left (P^{(m-p)}\Big\vert_{x_{3}=0}\right ),\sigma_p\left (P^{(m-p+1)}\Big\vert_{x_{3}=0}\right ),\dots, \sigma_p\left (P^{(m)}\Big\vert_{x_{3}=0}\right )\right ).
	\end{equation*}
\end{definition}

%

We notice that when the metrics $(\hat{\varepsilon},\hat{\mu})$,$(\hat{\varepsilon}',\hat{\mu}')$ are equal at $x_{3}=0$, the principal symbols of $B$ and $B'$ agree at the boundary. Hence $Y$ $\in$ $\Psi DO^{(0)}(\partial M)$, or in terms of the class of pseudodifferential operators given in Definition \ref{familyofPsiDosdef}, $Y=x_{3}Y^{(1)}+Y^{(0)}$,
where, restricted to $\partial M$, $Y^{(1)}\in\Psi DO^{(1)}(\partial M)$ and $Y^{(0)}\in\Psi DO^{(0)}(\partial M)$. This is equivalent to $Y \in \Psi DO^{(1,1)}(\partial M,\mathbb{R}^+)$ and correspondingly, for the pseudodifferential operator $J$, we have $J\in \Psi DO^{(1,1)}(\partial M,\mathbb{R}^+)$. Also, defining $\tilde{H}=H'-H$ we have, by Lemma \ref{prinsymbEE'HH'equallemma}, that $\tilde{H}=0$ at $x_{3}=0$ which allows us to write $\tilde{H}\in x_3\Psi DO^{(0,0)}(\partial M,\mathbb{R}^+) \subset \Psi DO^{(0,1)}(\partial M,\mathbb{R}^+)$. Similarly, writing $\tilde{E}=E'-E$, then $\tilde{E}$ is in the same space.

Aiming to prove Theorem \ref{normalderivativesthm} inductively, we calculate the normal derivatives of order $a$, $a\geq 1$, of $\tilde{H}$ and $\tilde{E}$ at $x_{3}=0$. In Theorems \ref{Dx3Hthm} and \ref{Dx3Ethm} we express these derivatives in terms of the pseudodifferential operators $B$, $Y$ and $C$, $J$, respectively. In Theorems \ref{Ythm} and \ref{Jthm} we use the Maxwell's and divergence equations to write $D_{x_{3}}^{a}\tilde{H}$ and $D_{x_{3}}^{a}\tilde{E}$ in terms of the normal derivatives of $\hat{\varepsilon}$, $\hat{\mu}$, $\hat{\varepsilon}'$, $\hat{\mu}'$ in boundary normal coordinates for $\hat{\varepsilon}$/$\hat{\varepsilon}'$ at $x_{3}=0$. Then, we use the expressions derived for $D_{x_{3}}^{a}\tilde{H}$ and $D_{x_{3}}^{a}\tilde{E}$ at $x_{3}=0$ to determine the vectors of principal symbols of the Riccati equations \eqref{YRiccati} and \eqref{JRiccati}, in the sense of Definition \ref{princsymbdef}.

\begin{theorem}\label{Dx3Hthm}
	Let $(\hat{\varepsilon},\hat{\mu})$, $(\hat{\varepsilon}',\hat{\mu}')$ be two different sets of electromagnetic parameters and $H$, $H'$ the associated magnetic fields, respectively and let $\tilde{H}=H'-H$. Assume that $B$, $B'$ are the factorisation operators corresponding to the magnetic fields $H$, $H'$ and let $\tilde{H} = H'-H$, $Y=B'-B$. Let us fix boundary normal coordinates for $\hat{\varepsilon}$/$\hat{\varepsilon}'$. If $Y \in \Psi DO^{(1,\kappa)}(\partial M,\mathbb{R}^+)$ and $\tilde{H}=0$ at $x_{3}=0$,
then we have $\tilde{H} \in \Psi DO^{(0,\kappa)}(\partial M,\mathbb{R}^+)$
	and for $1\leq a$
	\begin{equation}\label{Dx3aH}
	\sigma_{p}\left (D_{x_{3}}^{a}\tilde{H} \right )= \sum_{j=0}^{a-1}\sum_{n=0}^{j} \left(\begin{matrix}
	j\\
	n
	\end{matrix}\right)(B^{(1)})^{a-1-j}\sigma_{p}\left (D_{x_{3}}^{j-n}Y \right )(B^{(1)})^{n}H^{(0)}
	\end{equation}
	at $x_{3}=0$.
\end{theorem}

\begin{proof}
Making use of Theorem \ref{Dx3Htheorem} we have that
	\begin{equation} \label{Indc11}
	D_{x^3} \tilde{H} = YH+ B\tilde{H}+ Y\tilde{H}
	\end{equation}
modulo smoothing. In the case $\kappa = 1$, by taking the principal symbol of \eqref{Indc11} the proof is complete. So assume that $\kappa>1$. Starting from this, we prove inductively that for $1\leq a \leq \kappa$,
\begin{equation}\label{Indc1}
D_{x^3}^a \tilde{H} = \sum_{j=0}^{a-1} B^{a-1-j} D_{x_3}^j(YH) + B^a \tilde{H} +\tilde{R}_a \tilde{H} + R_a
\end{equation}
where $\tilde{R}_a \in \Psi DO^{(a-1,0)}(\partial M,\mathbb{R}^+)$ and $R_a \in \Psi DO^{(a,\kappa+1)}(\partial M,\mathbb{R}^+)$. Note that the case $a =1$ of \eqref{Indc1} is already proven from \eqref{Indc11} with $\tilde{R}_1 = 0$ and $R_1 = Y\tilde{H}$. 
	
Now assume, by induction, that \eqref{Indc1} is true for some $1\leq a$. Using Pascal's triangle we obtain 
\begin{equation}\label{Dx3Y}
D_{x_{3}}^{j}(YH)=\sum_{n=0}^{j}\left(\begin{matrix}
j\\
n
\end{matrix}\right) (D_{x_{3}}^{j-n}Y) B^{n}H,
\end{equation}
and applying Lemma \ref{PsiDOcalclem} to each term separately conclude that $D_{x_{3}}^{j}(YH) \in \Psi DO^{j+1,\kappa}(\partial M,\mathbb{R}^+)$. Therefore, applying $D_{x^3}$ to \eqref{Indc1}
gives
\[
D_{x^3}^{a+1} \tilde{H} = \sum_{j=0}^{a-1} B^{a-1-j} D_{x_3}^{j+1}(YH) + B^a D_{x^3} \tilde{H} +\tilde{R}_{a+1} \tilde{H}+R_{a+1}
\]
where  $\tilde{R}_{a+1} \in \Psi DO^{(a,0)}(\partial M,\mathbb{R}^+)$ and $R_{a+1} \in \Psi DO^{(a+1,\kappa+1)}(\partial M, \mathbb{R}^+)$ contain the derivatives of the corresponding terms from the $a$ step, as well as terms involving derivatives of $B$. Using \eqref{Indc11}, rearranging the terms and changing the remainder $R_{a+1}$, we see that \eqref{Indc1} holds for $a+1$. This completes the induction and so proves \eqref{Indc1}.

Using \eqref{Indc1} and \eqref{Dx3Y} and the fact that $\tilde{H}$ vanishes at the boundary, the principal symbol of $D^a_{x^3} \tilde{H}$ at $x_{3}=0$ is given by \eqref{Dx3aH}. This implies that, at $x_3 = 0$, $D^a_{x^3}\tilde{H} \in \Psi DO^{(a-\kappa)}(\partial M)$ and so, since this is true for $1 \leq a \leq \kappa$, we get $\tilde{H} \in \Psi DO^{(0,\kappa)}(\partial M,\mathbb{R}^+)$.
\end{proof}

\begin{theorem}\label{Dx3Ethm}
	Let $(\hat{\varepsilon},\hat{\mu})$, $(\hat{\varepsilon}',\hat{\mu}')$ be two different sets of electromagnetic parameters and $E$, $E'$ the associated electric fields, respectively and let $\tilde{E}=E'-E$. Assume that $C$, $C'$ are the factorisation operators corresponding to the magnetic fields $E$, $E'$ and let $J=C'-C$. If $J \in \Psi DO^{(1,\kappa)}(\partial M)$ and $\tilde{E}=0$ at $x_{3}=0$, then have $\tilde{E} \in  \Psi DO^{(0,\kappa)}(\partial M,\mathbb{R}^+)$ and for $1\leq a$
	\begin{equation}\label{Dx3aE}
	\sigma_{p}(D_{x_{3}}^{a}\tilde{E})= \sum_{j=0}^{a-1}\sum_{n=0}^{j} \left(\begin{matrix}
	j\\
	n
	\end{matrix}\right) (C^{(1)})^{a-1-j} \sigma_{p}(D_{x_{3}}^{j-n}J)(C^{(1)})^{n}E^{(0)},
	\end{equation}
	at $x_{3}=0$.
\end{theorem}

\begin{proof}
	The proof is done in a similar manner to the proof of Theorem \ref{Dx3Hthm} and therefore is omitted.
\end{proof}


We now continue to calculate the difference of the normal derivatives of Maxwell's and divergence equations for the two sets of parameters.

\begin{theorem}\label{Ythm}
	Let $(\hat{\varepsilon},\hat{\mu})$ and $(\hat{\varepsilon}',\hat{\mu}')$ be two different sets of electromagnetic parameters and $\Lambda_{\hat{\varepsilon}}$, $\Lambda_{\hat{\varepsilon}'}$, $H$, $H'$, $E$, $E'$ the associated impedance maps and electromagnetic fields, respectively and $\tilde{H}=H'-H$, $\tilde{E}=E'-E$.
	Assume that $B,B'$ are the factorisation operators corresponding to the magnetic fields $H$, $H'$ and let $Y=B'-B$. Assume further that $C$, $C'$ are the factorisation operators corresponding to the electric fields $E$, $E'$ and $J=C'-C$. Let us fix boundary normal coordinates for $\hat{\varepsilon}$/$\hat{\varepsilon}'$. If, for some $\kappa \geq 1$,
	\begin{itemize}
		\item $\Lambda_{\hat{\varepsilon}}=\Lambda_{\hat{\varepsilon}'}$,
		\item $\hat{\varepsilon}=\hat{\varepsilon}'+x_{3}^{\kappa}e_{\varepsilon}, \quad \hat{\mu}=\hat{\mu}'+x_{3}^{\kappa}e_{\mu},$ where $e_{\varepsilon}, e_{\mu}$ are symmetric,
		\item $\sqrt{\vert \hat{\varepsilon}^{-1}\vert} = \sqrt{\vert \hat{\varepsilon}'^{-1}\vert} + x_3^\kappa r_\varepsilon$, $\sqrt{\vert \hat{\mu}^{-1}\vert} = \sqrt{\vert \hat{\mu}'^{-1}\vert} + x_3^\kappa r_\mu$,
		\item $Y$, $J$ $\in$ $\Psi DO^{(1,\kappa)}(\partial M,\mathbb{R}^+)$,
	\end{itemize}
	then,
	\begin{eqnarray}\label{Y0}
				\sigma_{p}\left(D_{x_{3}}^{\kappa}\tilde{H}_{\tilde{c}}\right)&=&0, \quad \sigma_{p}\left(D_{x_{3}}^{\kappa}\tilde{H}_{3}\right)= (-i)^{\kappa} \kappa! \left (e_{\mu}^{3j} + \frac{r_\mu\hat{\mu}'^{3j}}{\sqrt{\vert\hat{\mu}'^{-1} \vert}}\right )\frac{H_{j}^{(0)}}{\hat{\mu}'^{33}},\\ 
			\label{Dx3kappa+1Htang}
			\sigma_{p}(D_{x_{3}}^{\kappa+1}\tilde{H}_{\tilde{c}})&=&(-i)^{\kappa} \kappa! \xi_{\tilde{c}}\left (e_{\mu}^{3j} + \frac{r_\mu\hat{\mu}'^{3j}}{\sqrt{\vert\hat{\mu}'^{-1} \vert}}\right )\frac{H_{j}^{(0)}}{\hat{\mu}'^{33}},\\
			\label{Dx3kappa+1Hmu}
			\hat{\mu}'^{3j}\sigma_{p}\left(D_{x_{3}}^{\kappa+1}\tilde{H}_{j}\right)&=& (-i)^{\kappa} \kappa!  \Bigg (- \frac{\hat{\mu}'^{\tilde{i}3}\xi_{\tilde{i}}}{\hat{\mu}'^{33}}\left (e_{\mu}^{3j} + \frac{r_\mu\hat{\mu}'^{3j}}{\sqrt{\vert\hat{\mu}'^{-1} \vert}}\right )\\
			\notag &
			& \hskip1cm +\  e_{\mu}^{\tilde{i}j}\xi_{\tilde{i}}+ 2 e_{\mu}^{3l}(B^{(1)})_{l}^{j} + \frac{r_\mu\hat{\mu}'^{3l}}{\sqrt{\vert\hat{\mu}'^{-1} \vert}}(B^{(1)})_{l}^{j}\Bigg)H_{j}^{(0)},\quad \quad
	\end{eqnarray}
	at $x_3 = 0$ and if $\kappa \geq 2$ then
	\begin{equation}
	\label {Dx3aH=0}
			\sigma_{p}(D_{x_{3}}^{a}\tilde{H}_{j})=0, \quad 1 \leq a \leq \kappa-1.
	\end{equation}
at $x_{3}=0$.

\end{theorem}
\begin{proof}
	First note that, by Theorem \ref{Dx3Hthm}, $\tilde{H} \in \Psi DO^{(0,\kappa)}(\partial M,\mathbb{R}^+)$. Let us first consider the divergence equations for the magnetic fields $H$ and $H'$, which are
	\begin{equation*}
	-i \delta_{\hat{\mu}}( H)=0, \quad -i\delta_{\hat{\mu}'}(H')=0.
	\end{equation*}
	Subtracting these equations and writing the result in  boundary normal coordinates for $\hat{\varepsilon}$/$\hat{\varepsilon}'$, we obtain
	\begin{equation}\label{divmu1kappa}
	\begin{split}
	\hat{\mu}'^{3j}D_{x_{3}}\tilde{H}_{j}&=- \hat{\mu}'^{\tilde{i}j}D_{x_{\tilde{i}}}\tilde{H}_{j}+ x_{3}^{\kappa}e_{\mu}^{\tilde{i}j}D_{x_{\tilde{i}}}H_{j}+ x_{3}^{\kappa}e_{\mu}^{3j}D_{x_{3}}H_{j}\\
	&\hskip1cm -i \kappa x_{3}^{\kappa-1}\left (e_{\mu}^{3j} + \frac{r_\mu\hat{\mu}'^{3j}}{\sqrt{\vert\hat{\mu}'^{-1} \vert}}\right )H_{j} + R
	\end{split}
	\end{equation}
where $R \in \Psi DO^{(0,\kappa)}(\partial M, \mathbb{R}^+)$ and the other terms on the right are in $\Psi DO^{(1,\kappa)}(\partial M, \mathbb{R}^+)$.
	Applying $D_{x_{3}}^{a}$ for $a \geq 0$ to \eqref{divmu1kappa}, we arrive at
	\begin{equation}\label{Dx3a}
	\begin{split}
	\hat{\mu}'^{3j}D_{x_{3}}^{a+1}\tilde{H}_{j}&= D_{x_{3}}^{a}\Bigg (- \hat{\mu}'^{\tilde{i}j}D_{x_{\tilde{i}}}\tilde{H}_{j}+ x_{3}^{\kappa}e_{\mu}^{\tilde{i}j}D_{x_{\tilde{i}}}H_{j}+ x_{3}^{\kappa}e_{\mu}^{3j}D_{x_{3}}H_{j}\\
	&-i \kappa x_{3}^{\kappa-1}\left (e_{\mu}^{3j} + \frac{r_\mu\hat{\mu}'^{3j}}{\sqrt{\vert\hat{\mu}'^{-1} \vert}}\right )H_{j}\Bigg ),
	\end{split}
	\end{equation}
	modulo $\Psi DO^{(a,\kappa)}(\partial M)$. Each of the terms in the right hand-side of \eqref{Dx3a} expands as
	\begin{equation}\label{Dx3aeqs}
	\begin{split}
	D_{x_{3}}^{a}(- \mu'^{\tilde{i}j}D_{x_{\tilde{i}}}\tilde{H}_{j})&= - \hat{\mu}'^{\tilde{i}j}D_{x_{\tilde{i}}}D_{x_{3}}^{a}\tilde{H}_{j},\\
	D_{x_{3}}^{a}( x_{3}^{\kappa}e_{\mu}^{\tilde{i}j}D_{x_{\tilde{i}}}H_{j})&= \sum_{n=0}^{a} \left(\begin{matrix}
	a\\
	n
	\end{matrix}\right)(D_{x_{3}}^{n}x_{3}^{\kappa})e_{\mu}^{\tilde{i}j}D_{x_{\tilde{i}}} D_{x_{3}}^{a-n}H_{j},\\
	D_{x_{3}}^{a}(x_{3}^{\kappa}e_{\mu}^{3j}D_{x_{3}}H_{j})&= \sum_{n=0}^{a} \left(\begin{matrix}
	a\\
	n
	\end{matrix}\right) (D_{x_{3}}^{n}x_{3}^{\kappa})e_{\mu}^{3j}D_{x_{3}}^{a+1-n}H_{j},\\
	D_{x_{3}}^{a}\left (-i \kappa x_{3}^{\kappa-1}\left (e_{\mu}^{3j} + \frac{r_\mu\hat{\mu}'^{3j}}{\sqrt{\vert\hat{\mu}'^{-1} \vert}}\right )H_{j}\right )&=\\
	&\hskip-2.3cm -i\kappa\sum_{n=0}^{a} \left(\begin{matrix}
	a\\
	n
	\end{matrix}\right)  (D_{x_{3}}^{n}x_{3}^{\kappa-1})\left (e_{\mu}^{3j} + \frac{r_\mu\hat{\mu}'^{3j}}{\sqrt{\vert\hat{\mu}'^{-1} \vert}}\right )D_{x_{3}}^{a-n}H_{j},
	\end{split}
	\end{equation}
	modulo $\Psi DO^{(a,\kappa)}(\partial M)$.
	Let us consider the case of $\kappa=1$. The principal symbols of equation \eqref{Dx3a} for $a=0$ and $a=1$ at $x_{3}=0$ give respectively
	\begin{equation}\label{divmu1}
	\begin{split}
	\hat{\mu}'^{3l}\sigma_{p}(D_{x_{3}}\tilde{H}_{j})H_{j}^{(0)}&=-i \left (e_{\mu}^{3j} + \frac{r_\mu\hat{\mu}'^{3j}}{\sqrt{\vert\hat{\mu}'^{-1} \vert}}\right )H_{j}^{(0)},\\
	\hat{\mu}'^{3j}\sigma_{p}(D_{x_{3}}^{2}\tilde{H}_{j})&= - \hat{\mu}'^{\tilde{i}k}\xi_{\tilde{i}}\sigma_{p}(D_{x_{3}}
	\tilde{H}_{j})\\
	&\hskip0.2cm+ (-i) \Bigg (  e_{\mu}^{\tilde{i}k}\xi_{\tilde{i}} +2 e_{\mu}^{3j}(B^{(1)})_{j}^{k} +  \frac{r_\mu\hat{\mu}'^{3j}}{\sqrt{\vert\hat{\mu}'^{-1} \vert}}(B^{(1)})_{j}^{k}\Bigg) H_{k}^{(0)}.
	\end{split}
	\end{equation}
Assuming that $\kappa \geq 2$ and considering the principal symbols of equation \eqref{Dx3a} for $1 \leq a \leq \kappa$ at $x_{3}=0$ we obtain
	\begin{equation}\label{divmukappa}
	\begin{split}
	\hat{\mu}'^{3j}\left(D_{x_{3}}^{l}\tilde{H}_{j}\right)&=-\hat{\mu}'^{\tilde{i}j}\xi_{\tilde{i}}\sigma_{p}(D_{x_{3}}^{l-1}\tilde{H}_{j}),  \quad 1 \leq l \leq \kappa-1, \\
	\hat{\mu}'^{3j}\sigma_{p}\left(D_{x_{3}}^{\kappa}\tilde{H}_{j}\right)&=- \hat{\mu}'^{\tilde{i}j}\xi_{\tilde{i}}\sigma_{p}(D_{x_{3}}^{\kappa-1}\tilde{H}_{j})+ (-i)^{\kappa} \kappa! \left (e_{\mu}^{3j} + \frac{r_\mu\hat{\mu}'^{3j}}{\sqrt{\vert\hat{\mu}'^{-1} \vert}}\right )H_{j}^{(0)},\\
	\hat{\mu}'^{3j}\sigma_{p}\left(D_{x_{3}}^{\kappa+1}\tilde{H}_{j}\right)&=- \hat{\mu}'^{\tilde{i}j}\xi_{\tilde{i}}\sigma_{p}(D_{x_{3}}^{\kappa}\tilde{H}_{j})\\
	& \hskip0.2cm+ (-i)^{\kappa} \kappa! \Bigg ( e_{\mu}^{\tilde{i}k}\xi_{\tilde{i}}+ 2 e_{\mu}^{3l}(B^{(1)})_{l}^{k} +  \frac{r_\mu\hat{\mu}'^{3j}}{\sqrt{\vert\hat{\mu}'^{-1} \vert}}(B^{(1)})_{j}^{k}\Bigg )H_{k}^{(0)},
	\end{split}
	\end{equation}
	where we have made use of equations \eqref{Dx3aeqs} for $1 \leq a \leq \kappa$ at $x_{3}=0$.
	
	In order to derive the equations regarding the tangential components of $\tilde{H}$ we start with Maxwell's equation
	\begin{equation*}
	\ast_{\hat{\varepsilon}}\mathrm{d}H= -i \omega E,
	\end{equation*}
	which is written in boundary normal coordinates for $\hat{\varepsilon}$/$\hat{\varepsilon}'$ as
	\begin{equation*}
	\sigma^{lbc}D_{x_{b}}H_{c}= \omega \frac{\hat{\varepsilon}^{lk}}{\sqrt{\lvert \hat{\varepsilon}\rvert}}E_{k}.
	\end{equation*}
	Evaluating the above at either $l=1$ or $l=2$ and taking into account that we work in boundary normal coordinates for $\hat{\varepsilon}$ implies
	\begin{equation}\label{imp1}
	D_{x_{3}}H_{\tilde{c}}= D_{x_{\tilde{c}}}H_{3}+ \omega \sigma_{\tilde{a}3\tilde{c}}\frac{\hat{\varepsilon}^{lk}}{\sqrt{\lvert \hat{\varepsilon}\rvert}}E_{\tilde{k}}.
	\end{equation}
	Subtracting the corresponding equation for the $'-$ parameters,
we arrive at
	\begin{equation*}
	D_{x_{3}}\tilde{H}_{\tilde{c}}=D_{x_{\tilde{c}}}\tilde{H}_3 + \omega\sigma_{\tilde{a}3\tilde{c}}\left (\frac{\hat{\varepsilon}'^{lk}}{\sqrt{\lvert \hat{\varepsilon}'\rvert}}E'_{\tilde{k}}-\frac{\hat{\varepsilon}^{lk}}{\sqrt{\lvert \hat{\varepsilon}\rvert}}E_{\tilde{k}}\right ).
	\end{equation*}
	Using Theorem \ref{Dx3Ethm} and the hypotheses, we obtain
modulo $\Psi DO^{(0,\kappa)}(\partial M)$
	\begin{equation*}
	\begin{split}
	D_{x_{3}}\tilde{H}_{\tilde{c}}&=D_{x_{\tilde{c}}}\tilde{H}_{3}.
	\end{split}
	\end{equation*}
	Applying $D_{x_{3}}$ $a$ more times leads to
	\begin{equation*}
	\begin{split}
	D_{x_{3}}^{a+1}\tilde{H}_{\tilde{c}}=D_{x_{\tilde{c}}}D_{x_{3}}^{a}\tilde{H}_{3}, \quad \mbox{modulo} \ \Psi DO^{(a,\kappa)}(\partial M)
	\end{split}
	\end{equation*}
	for $0\leq a \leq \kappa$. The principal symbol of the above equation at $x_{3}=0$ gives
	\begin{equation}\label{impeda}
	\sigma_{p}(D_{x_{3}}^{a+1}\tilde{H}_{\tilde{c}})=\xi_{\tilde{c}}\sigma_{p}(D_{x_{3}}^{a}\tilde{H}_{3}), \quad 0<a\leq \kappa.
	\end{equation}
	Combining equations \eqref{divmu1} and \eqref{divmukappa} with \eqref{impeda} completes the proof.
\end{proof}

\begin{theorem}\label{Jthm}
Assume the same hypotheses as for Theorem \ref{Ythm} and in addition
\begin{itemize}
\item $\hat{\varepsilon}=\hat{\varepsilon}'+x_{3}^{\kappa+1}e_{\varepsilon}$,
\item $\sqrt{\vert \hat{\varepsilon}^{-1}\vert} = \sqrt{\vert \hat{\varepsilon}'^{-1}\vert} + x_3^{\kappa+1} r_\varepsilon$.
\end{itemize}
Then
\begin{eqnarray}
\sigma_p(D_{x^3}^{\kappa+1} \tilde{E}_{\tilde{c}}) &=& 0, \quad  \sigma_{p}\left(D_{x_{3}}^{{\kappa+1}}\tilde{E}_{3}\right)= (-i)^{\kappa+1} (\kappa+1)! \frac{r_\varepsilon}{\sqrt{\vert\hat{\varepsilon}'^{-1} \vert}}E_{3}^{(0)},\label{Dx3kappa+1Eeqs}\\
 \sigma_{p}\left(D_{x_{3}}^{\kappa+2}\tilde{E}_{\tilde{c}}\right)&=& (-i)^{\kappa+1} (\kappa+1)! \xi_{\tilde{c}} \frac{r_\varepsilon}{\sqrt{\vert\hat{\varepsilon}'^{-1} \vert}}E_{3}^{(0)},\\
\sigma_{p}\left(D_{x_{3}}^{\kappa+2}\tilde{E}_{3}\right)&=& (-i)^{\kappa+1} (\kappa+1)!  \Bigg (  e_{\varepsilon}^{\tilde{i}j}\xi_{\tilde{i}}+ \frac{r_\varepsilon}{\sqrt{\vert\hat{\varepsilon}'^{-1} \vert}}(C^{(1)})_{3}^{j}\Bigg)E_{j}^{(0)},\quad \quad \\
\sigma_{p}(D_{x_{3}}^{a}\tilde{E}_{j})&=&0, \quad 1 \leq a \leq \kappa \label{Dx3aE=0}
\end{eqnarray}
at $x_3 = 0$.
%
\end{theorem}
\begin{proof}
Working as in the proof of Theorem \ref{Ythm} leads to the desired result. Note that the calculations and results are simpler in this case because of the choice of boundary normal coordinates.
\end{proof}

\subsection{Proof of Theorem \ref{normalderivativesthm}}\label{normalderivativesthmsec}
To prove Theorem \ref{normalderivativesthm}, it suffices to prove Theorems \ref{indHthm}, \ref{indEthm} and \ref{orderofYJthm} below. More specifically, in each inductive step we start with Theorem \ref{indHthm} to prove given that the normal derivatives of order $\kappa-1 \geq 0$ of $\hat{\varepsilon}$ and $\hat{\varepsilon}'$ agree at $x_{3}=0$ the same follows for their normal derivatives of order $\kappa$ at $x_{3}=0$. Then, we proceed to Theorem \ref{indEthm} to prove the corresponding result for the metrics $\hat{\mu}$ and $\hat{\mu}'$. Last, we have to show that as the normal derivatives of the metrics induced by the electromagnetic parameters agree to a higher order at $x_{3}=0$ the order of the pseudodifferential operators $Y$ and $J$ decreases at $x_{3}=0$. This is shown in Theorem \ref{orderofYJthm}.

In what follows, we will work with specific values of the principal symbols of $H$ and $E$ by choosing $H^{(0)}=\chi_{\hat{\varepsilon}}$ and $E^{(0)}=\chi_{\hat{\mu}}$. Recall Section \ref{systemdecouplingsec} for the definitions of these covectors. These covectors are not defined when $\xi_{\hat{\varepsilon}3} = \xi_{\hat{\mu}3}$, but in these cases the conclusions still follow by density since, as observed in Remark \ref{basisremark}, the symbols are continuous with respect to the parameters.
\begin{theorem}\label{indHthm}
	Let $(\hat{\varepsilon},\hat{\mu})$ and $(\hat{\varepsilon}',\hat{\mu}')$ be two different sets of electromagnetic parameters and $\Lambda_{\hat{\varepsilon}}$, $\Lambda_{\hat{\varepsilon}'}$ the associated impedance maps. Assume that $B$, $B'$ are the factorisation operators corresponding to the magnetic fields $H$, $H'$ and let $Y=B'-B$. Let us fix boundary normal coordinates for $\hat{\varepsilon}$/$\hat{\varepsilon}'$. If
	\begin{equation*}
	\Lambda_{\hat{\varepsilon}}=\Lambda_{\hat{\varepsilon}'}, 
	\end{equation*}
	\begin{equation*}
	\partial_{x_{3}}^{a}\hat{\varepsilon}=\partial_{x_{3}}^{a}\hat{\varepsilon}', \quad \partial_{x_{3}}^{a}\hat{\mu}=\partial_{x_{3}}^{a}\hat{\mu}', \quad \mbox{at} \ x_{3}=0,
	\end{equation*}
	for $a=0,\dots, \kappa-1$
	and 
	\begin{equation*}
	Y \in \Psi DO^{(1,\kappa)}(\partial M,\mathbb{R}^+),
	\end{equation*}
	then
	\begin{equation*} \partial_{x_{3}}^{\kappa}\hat{\varepsilon}=\partial_{x_{3}}^{\kappa}\hat{\varepsilon}',
	\end{equation*}	
	at $x_{3}=0$.
\end{theorem}
\begin{proof}
	Since $Y\in\Psi DO^{(1,\kappa)}(\partial M,\mathbb{R}^+)$, both sides of Riccati equation \eqref{YRiccati} are in $\Psi DO^{(2,\kappa)}(\partial M,\mathbb{R}^+)$ and the vector of principal symbols, in the sense of Definition \ref{princsymbdef}, is given by
	\begin{equation} \label{vectorkH}
	\begin{split}
	&A_{H'} Y^{(1)}+ T_{H'}B^{(1)}Y^{(1)}+T_{H'}Y^{(1)}B^{(1)}=\\
	& \hskip0.5cm \frac{1}{\kappa!}\left(\partial_{x_{3}}^{\kappa}(A_{H}-A_{H'})B^{(1)}+ \partial_{x_{3}}^{\kappa}(T_{H}-T_{H'})(B^{(1)})^{2} + \partial_{x_{3}}^{\kappa}(Q_{H}-Q_{H'})\right),\\
	&-\kappa i T_{H'}Y^{(1)}+	A_{H'}Y^{(0)}+ T_{H'}B^{(1)}Y^{(0)} + T_{H'}Y^{(0)}B^{(1)}=\\
	& \hskip3cm \frac{1}{(\kappa-1)!}\left(\partial_{x^{3}}^{\kappa-1}(G_{H}-G_{H'})B^{(1)} + \partial_{x^{3}}^{\kappa-1}(F_{H}-F_{H'})\right),\\
&-(\kappa-1)i	T_{H'}Y^{(0)}+	A_{H'}Y^{(-1)}+T_{H'}B^{(1)}Y^{(-1)}+ T_{H'}Y^{(-1)}B^{(1)}=\\
&\hskip7.5cm \frac{\partial_{x_{3}}^{\kappa-2}}{(\kappa-2)!}(R_{H}-R_{H'}),\\
&-(\kappa-2) i T_{H'}Y^{(-1)}+A_{H'} Y^{(-2)}+ T_{H'}B^{(1)}Y^{(-2)}+T_{H'}Y^{(-2)} B^{(1)}=0,\\
&\hskip5cm\vdots
\\
& -i T_{H'}Y^{-(\kappa-2)}+ A_{H'}Y^{-(\kappa-1)}+ T_{H'}B^{(1)}Y^{-(\kappa-1)}+ T_{H'}Y^{-(\kappa-1)}B^{(1)}=0.
\end{split}
\end{equation}
The different sets of electromagnetic parameters are assumed to agree to order $\kappa-1$ at $x_{3}=0$, which implies that
\begin{equation*}
\hat{\varepsilon}^{-1}= \hat{\varepsilon}'^{-1}+ x_{3}^{\kappa}\tilde{e}_{\varepsilon}, \quad \hat{\mu}= 	\hat{\mu}'+ x_{3}^{\kappa}e_{\mu}.
\end{equation*}
To prove the theorem, we aim to prove that $\tilde{e}_{\varepsilon}=0$ at $x_{3}=0$ beginning by
applying the second equation of \eqref{vectorkH} to $H^{(0)} = \chi_{\hat{\epsilon}}$. If $\kappa=1$, Lemma \ref{GFHlemma} is directly satisfied which implies the desired result. If $\kappa\geq 2$, in order for the assumption of Lemma \ref{GFHlemma} for $Y$ to be satisfied we additionally use \eqref{Dx3aH} and \eqref{Dx3aH=0} which when combined give
\begin{equation*}
Y^{(l)}\chi_{\varepsilon}=0, \quad 1-\kappa \leq l \leq -1.
\end{equation*}
at $x_{3}=0$.
Therefore the assumptions of Lemma \ref{GFHlemma} are satisfied for any $\kappa \geq 1$ and the proof is completed.
\end{proof}

\begin{theorem}\label{indEthm}
	Let $(\hat{\varepsilon},\hat{\mu})$, $(\hat{\varepsilon}',\hat{\mu}')$ be two different sets of electromagnetic parameters and $\Lambda_{\hat{\varepsilon}}$, $\Lambda_{\hat{\varepsilon}'}$ the associated impedance maps. Assume that $C$, $C'$ are the  factorisation operators corresponding to the electric fields $E$, $E'$ and let $J=C'-C$.  Let us fix boundary normal coordinates for $\hat{\varepsilon}$/$\hat{\varepsilon}'$. If
	\begin{equation*}
	\Lambda_{\hat{\varepsilon}}=\Lambda_{\hat{\varepsilon}'},
	\end{equation*}
	\begin{equation*}  
	\partial_{x_{3}}^{a}\hat{\varepsilon} = \partial_{x_{3}}^{a}\hat{\varepsilon}', \quad \partial_{x_{3}}^{b}\hat{\mu}=\partial_{x_{3}}^{b}\hat{\mu}', \quad \mbox{at} \ x_{3}=0,
	\end{equation*}
	for $a=0,1,2,\dots,\kappa$, $b=0,1,\dots,\kappa-1$
	and
	\begin{equation*}
	J \in \Psi DO^{(1,\kappa)}(\partial M),
	\end{equation*}
	then
	\begin{equation}
	\partial_{x_{3}}^{\kappa}\hat{\mu}=\partial_{x_{3}}^{\kappa}\hat{\mu}', \quad \mbox{at} \ x_{3}=0.
	\end{equation}
\end{theorem}
\begin{proof}
Given that $J$ $\in$ $\Psi DO^{(1,\kappa)}(\partial M)$ the principal symbol of the Riccati equation \eqref{JRiccati} at $x_{3}=0$ is given by
\begin{equation}\label{vectorkE}
\begin{split}
&A_{E} J^{(1)}+ T_{E}C^{(1)}J^{(1)}+T_{E}J^{(1)}C^{(1)}=\\
&\hskip1cm \frac{1}{\kappa!}\left(\partial_{x_{3}}^{\kappa}(A_{E}-A_{E'})C^{(1)}+ \partial_{x_{3}}^{\kappa}(T_{E}-T_{E'})(C^{(1)})^{2} + \partial_{x_{3}}^{\kappa}(Q_{E}-Q_{E'})\right),\\
&-\kappa i T_{E}J^{(1)}+	A_{E}J^{(0)}+ T_{E}C^{(1)}J^{(0)} + T_{E}J^{(0)}C^{(1)}=\\
&\hskip3.3cm \frac{1}{(\kappa-1)!}\left(\partial_{x^{3}}^{\kappa-1}(G_{E}-G_{E'})C^{(1)} + \partial_{x^{3}}^{\kappa-1}(F_{E}-F_{E'})\right),\\
&-(\kappa-1)i	T_{E}J^{(0)}+	A_{E}J^{(-1)}+T_{E}C^{(1)}J^{(-1)}+ T_{E}J^{(-1)}C^{(1)}\\
&\hskip4cm =\frac{1}{(\kappa-2)!}\partial_{x_{3}}^{\kappa-2}(R_{E}-R_{E'}),\\
&-(\kappa-2) i T_{E}J^{(-1)}+A_{E} J^{(-2)}+ T_{E}C^{(1)}J^{(-2)}+T_{E}J^{(-2)} C^{(1)}=0,\\
&\hskip5cm\vdots
\\
& -i T_{E}J^{-(\kappa-2)}+ A_{E}J^{-(\kappa-1)}+ T_{E}C^{(1)}J^{-(\kappa-1)}+ T_{E}J^{-(\kappa-1)}C^{(1)}=0.
\end{split}
\end{equation}
According to the assumptions of the theorem, the different sets of metrics satisfy
\begin{equation*}
\hat{\varepsilon}=\hat{\varepsilon}'+ x_{3}^{\kappa+1}e_{\varepsilon}, \quad (\hat{\mu}^{-1})=(\hat{\mu}'^{-1})+ x_{3}^{\kappa}\tilde{e}_{\mu},
\end{equation*}
for $e_{\varepsilon}$, $\tilde{e}_{\mu}$ symmetric.
Let us fix $E^{(0)}=\chi_{\hat{\mu}}$.
In order for Lemma \ref{GFElemma} to be applicable, we will show that for any $\kappa \geq 1$ we have
\begin{equation}\label{J^leq}
J^{(l)}\chi_{\hat{\mu}}=0, \quad 1-\kappa \leq l \leq 1,
\end{equation}
at $x_{3}=0$.
Equations \eqref{Dx3aE} and \eqref{Dx3aE=0} show that for any $\kappa \geq 1$
\begin{equation*}
 J^{(l)}\chi_{\hat{\mu}}=0,\quad 1-\kappa\leq l \leq 0 \quad \mbox{at} \ x_{3}=0.
 \end{equation*}
To show that $J^{(1)}$ also vanishes, we consider the second equation in \eqref{vectorkE}, \eqref{Tsym}, \eqref{Gsym}, \eqref{Fsym}, \eqref{Dx3aE} and \eqref{Dx3kappa+1Eeqs}. Using all of these formulas and taking the $\hat{\mu}'$ inner product with $\nu$, we obtain that $r_\epsilon$ appearing in \eqref{Dx3kappa+1Eeqs} must be zero. This implies, again by \eqref{Dx3aE} and \eqref{Dx3kappa+1Eeqs}, that $J^{(1)}\chi_{\hat{\mu}} = 0$. Hence, we obtain that equation \eqref{J^leq} holds for any $\kappa \geq 1$ which allows us to apply Lemma \ref{GFElemma} and the proof is completed.
\end{proof}

The following two lemmas are used in the proofs of Theorems \ref{Ythm} and \ref{Jthm} and include many of the technical calculations.

\begin{lemma}\label{GFHlemma}
	Let $(\hat{\varepsilon},\hat{\mu})$, $(\hat{\varepsilon}',\hat{\mu}')$ be two different sets of electromagnetic parameters. Let us fix boundary normal coordinates for $\hat{\varepsilon}$/ $\hat{\varepsilon}'$. If
	\begin{equation}\label{Lem7hyp}
	\hat{\varepsilon}^{-1}= \hat{\varepsilon}'^{-1}+ x_{3}^{\kappa}\tilde{e}_{\varepsilon}, \quad \hat{\mu}= 	\hat{\mu}'+ x_{3}^{\kappa}e_{\mu}
	\end{equation}
	and $Y$ $\in$ $\Psi DO^{(1,\kappa)}(\partial M)$ such that for any $\kappa\geq2$
	\begin{equation*}
	Y^{(a)}\chi_{\hat{\varepsilon}}=0, \quad 1-\kappa \leq a \leq -1,
	\end{equation*}
	then for any $\kappa \geq 1$, equation
	\begin{equation}\label {GFkappa}
	\begin{split}
&\left(T_{H'}(	-\kappa i Y^{(1)}+ B^{(1)}Y^{(0)} + Y^{(0)}B^{(1)})+A_{H'}Y^{(0)}\right)\chi_{\hat{\varepsilon}}=\\
&\hskip2cm \frac{1}{(\kappa-1)!}\left(\partial_{x^{3}}^{\kappa-1}(G_{H}-G_{H'})B^{(1)}+ \partial_{x^{3}}^{\kappa-1}(F_{H}-F_{H'})\right)\chi_{\hat{\varepsilon}}
\end{split}
	\end{equation}
	implies that
	\begin{equation*} \partial_{x_{3}}^{\kappa}\hat{\varepsilon}=\partial_{x_{3}}^{\kappa}\hat{\varepsilon}',
\end{equation*}	
\end{lemma}
\begin{proof}
We first show that for any $\kappa \geq 1$,
\begin{equation}\label{Lem71}
(Y^{(0)})_{\tilde{l}}^{k}{\chi_{\hat{\varepsilon}}}_{k}=0, \quad  (Y^{(0)})_{3}^{k}{\chi_{\hat{\varepsilon}}}_{k}=-i \kappa \Bigg (
\frac{\langle \nu, \chi_{\hat{\varepsilon}}\rangle_{e_{\mu}}}{\hat{\mu}'^{33}}+\frac{r_\mu \langle \nu, \chi_{\hat{\varepsilon}}\rangle_{\hat{\mu}'}}{\sqrt{\vert \hat{\mu}'^{-1} \vert}\hat{\mu}'^{33}} \Bigg ),
\end{equation}
at $x_{3}=0$. Evaluating equation \eqref{Dx3aH} for $a=\kappa$ at $H^{(0)}=\chi_{\varepsilon}$ results to
 \begin{equation*}
 \sigma_{p}(D_{x_{3}}^{\kappa}\tilde{H})\vert_{H^{(0)}=\chi_{\hat{\varepsilon}}}= \sum_{j=0}^{\kappa-1}\sum_{n=0}^{j} \left(\begin{matrix}
 j\\
 n
 \end{matrix}\right) (B^{(1)})^{\kappa-1-j} \sigma_{p}(D_{x_{3}}^{j-n}Y)(B^{(1)})^{n}\chi_{\hat{\varepsilon}}.
 \end{equation*}
 The assumption of the lemma on $Y$ implies that 
 \begin{equation*}
 \sigma_{p}(D_{x_{3}}^{\kappa}\tilde{H})\vert_{H^{(0)}=\chi_{\hat{\varepsilon}}}= \sigma_{p}(D_{x_{3}}^{\kappa-1}Y)\chi_{\hat{\varepsilon}}=(-i)^{\kappa-1}(\kappa-1)!Y^{(0)}\chi_{\hat{\varepsilon}},
 \end{equation*}
 at $x_{3}=0$.
Equation \eqref{Y0} then implies \eqref{Lem71}.

Now let us write \eqref{Dx3aH} for $a=\kappa+1$ evaluated at $H^{(0)}=\chi_{\hat{\varepsilon}}$, which gives
	\begin{equation*}
	\sigma_{p}(D_{x_{3}}^{\kappa+1}\tilde{H})\vert_{H^{(0)}=\chi_{\hat{\varepsilon}}}= \sum_{j=0}^{\kappa}\sum_{n=0}^{j} \left(\begin{matrix}
 j\\
 n
 \end{matrix}\right) (B^{(1)})^{\kappa-j} \sigma_{p}(D_{x_{3}}^{j-n}Y)(B^{(1)})^{n}\chi_{\hat{\varepsilon}}.
	\end{equation*}
The assumption of the theorem on $Y$ implies that the only nonzero terms in the above arise for $j = \kappa$, $n=0, 1$ and $j=\kappa-1$, $n=0$. Thus, we have
	\begin{equation*}
	\sigma_{p}(D_{x_{3}}^{\kappa+1}\tilde{H})\vert_{H^{(0)}=\chi_{\hat{\varepsilon}}}= (-i)^{\kappa}\kappa!Y^{(1)}\chi_{\hat{\varepsilon}}+(-i)^{\kappa-1}(\kappa-1)! \left (B^{(1)}Y^{(0)}+Y^{(0)} B^{(1)} \right )\chi_{\hat{\varepsilon}}.
	\end{equation*}
	Hence, we observe that the left hand-side of \eqref{GFkappa} is equal to
	\begin{equation}\label{LHSkappa}
	T_{H'}\frac{	\sigma_{p}(D_{x_{3}}^{\kappa+1}\tilde{H})\vert_{H^{(0)}=\chi_{\hat{\varepsilon}}}}{(-i)^{\kappa-1}(\kappa-1)!}+ A_{H'}Y^{(0)}\chi_{\hat{\varepsilon}}.
	\end{equation}
	Using equations \eqref{Tsym}, \eqref{Asym}, \eqref{Dx3kappa+1Htang}, \eqref{Dx3kappa+1Hmu} and \eqref{Lem71}, the left hand side of \eqref{GFkappa}, as shown in \eqref{LHSkappa}, is equal to
		\begin{equation}\label{LHScoordkappa}
	- i \kappa \frac{\sqrt{\vert \hat{\mu}^{-1}\vert}}{\sqrt{\vert\hat{\varepsilon}^{-1}\vert}}\left ( (\hat{\varepsilon}'^{-1})_{lq}\hat{\mu}^{q3} \langle \xi_{\hat{\epsilon}}, \chi_{\varepsilon}\rangle_{e_{\mu}}+ (\hat{\varepsilon}'^{-1})_{lq}\hat{\mu}^{qa}\xi_{\hat{\epsilon}a} \left (e_{\mu}^{3j} + \frac{r_\mu\hat{\mu}'^{3j}}{\sqrt{\vert\hat{\mu}'^{-1} \vert}}\right )\chi_{\hat{\epsilon}j}  \right ).
	\end{equation}
Using formulas \eqref{Gsym}, \eqref{Fsym} and \eqref{Lem7hyp}, the right hand side of \eqref{GFkappa} is written in coordinates as
	\begin{equation}\label{RHScoordkappa}
	\begin{split}
	&- i \kappa \frac{\sqrt{\vert \hat{\mu}^{-1}\vert}}{\sqrt{\vert\hat{\varepsilon}^{-1}\vert}}\left ( (\hat{\varepsilon}'^{-1})_{lq}\hat{\mu}^{q3} \langle \xi_{\hat{\epsilon}}, \chi_{\varepsilon}\rangle_{e_{\mu}}+ (\hat{\varepsilon}'^{-1})_{lq}\hat{\mu}^{qa}\xi_{\hat{\epsilon}a} \left (e_{\mu}^{3j} + \frac{r_\mu\hat{\mu}'^{3j}}{\sqrt{\vert\hat{\mu}'^{-1} \vert}}\right )\chi_{\hat{\epsilon}j}  \right )\\
	&\hskip 3cm+i \kappa \frac{(\hat{\varepsilon}'^{-1})_{lq}}{\vert \hat{\varepsilon}'^{-1}\vert} \sigma^{3 \tilde{j}\tilde{q}} (\tilde{e}_{\varepsilon})_{\tilde{b}\tilde{j}} \sigma^{d k \tilde{b}} \xi_{\hat{\varepsilon}d} \chi_{\hat{\varepsilon}k}.
	\end{split}
	\end{equation}
	Equating \eqref{LHScoordkappa} with \eqref{RHScoordkappa} we arrive at
	\begin{equation*}
	i \kappa \frac{(\hat{\varepsilon}'^{-1})_{lq}}{\vert \hat{\varepsilon}'^{-1}\vert} \sigma^{3 \tilde{j}\tilde{q}} (\tilde{e}_{\varepsilon})_{\tilde{b}\tilde{j}} \sigma^{d k \tilde{b}} \xi_{\hat{\varepsilon}d} \chi_{\hat{\varepsilon}k}=0,
\end{equation*}
which, substituting a coordinate expression for $\chi_{\hat{\varepsilon}}$ and cancelling some terms, is simplified to
\begin{equation*}
(\varepsilon^{-1})_{\tilde{l}\tilde{l'}}\sigma^{3\tilde{\beta}\tilde{l'}}(\tilde{e}_{\varepsilon})_{\tilde{b}\tilde{\beta}}\varepsilon^{\tilde{b}\tilde{b'}}{\xi}_{\tilde{b'}}=0.
\end{equation*}
This holds for all $\tilde{\xi}$ and $l$ and so is satisfied if and only if $\tilde{e}_\varepsilon = 0$. Therefore $\hat{\varepsilon}^{-1}$ and $\hat{\varepsilon}'^{-1}$ agree up to the $\kappa$ derivative, and so the same is true for $\hat{\varepsilon}$ and $\hat{\varepsilon}'$. This completes the proof.
\end{proof}
\begin{lemma}\label{GFElemma}
	Let $(\hat{\varepsilon},\hat{\mu})$, $(\hat{\varepsilon}',\hat{\mu}')$ be two different sets of electromagnetic parameters. Let us fix boundary normal coordinates for $\hat{\varepsilon}$/$\hat{\varepsilon}'$. If
	\begin{equation*}
	\hat{\varepsilon}=\hat{\varepsilon}'+ x_{3}^{\kappa+1}e_{\varepsilon}, \quad (\hat{\mu}^{-1})=(\hat{\mu}'^{-1})+ x_{3}^{\kappa}\tilde{e}_{\mu}
	\end{equation*}
	and for any $\kappa \geq 1$ the following equations hold at $x_{3}=0$
	\begin{equation}\label{GFEeq}
	\left(\partial_{x^{3}}^{\kappa-1}(G_{E}-G_{E'})C^{(1)} + \partial_{x^{3}}^{\kappa-1}(F_{E}-F_{E'})\right)\chi_{\hat{\mu}}=0,
	\end{equation}
	\begin{equation}\label{AQTEeq}
	\left(\partial_{x_{3}}^{\kappa}(A_{E}-A_{E'})C^{(1)}+ \partial_{x_{3}}^{\kappa}(Q_{E}-Q_{E'})+ \partial_{x_{3}}^{\kappa}(T_{E}-T_{E'})(C^{(1)})^{2}\right)\chi_{\hat{\mu}}=0,
	\end{equation}
	then,
	\begin{equation*}
	\partial_{x_{3}}^{\kappa}\hat{\mu}=\partial_{x_{3}}^{\kappa}
	\hat{\mu}', \quad \mbox{at} \ x_{3}=0.
\end{equation*}
\end{lemma}
\begin{proof}
	Making use of the symbols $G_{E}$, $F_{E}$, which are given by equations \eqref{Gsym} and \eqref{Fsym} with $\hat{\varepsilon}$ and $\hat{\mu}$ switched, and the assumptions of the lemma on the different sets of metrics, equation \eqref{GFEeq} implies
	\begin{equation*}
	(\hat{\mu}^{-1})_{l\tilde{q}}\sigma^{3\tilde{j}\tilde{q}}(\tilde{e}_{\mu})_{\tilde{j}b}\sigma^{dkb}{\xi_{\hat{\mu}}}_{d}{\chi_{\hat{\mu}}}_{k}=0.
	\end{equation*}
	Substituting the expression of $\chi_{\hat{\mu}}$. the above reduces to
	\begin{equation*}
	(\hat{\mu}^{-1})_{l\tilde{q}}\sigma^{3\tilde{j}\tilde{q}}(\tilde{e}_{\mu})_{\tilde{j}b}\hat{\mu}^{bb'}{\xi_{\hat{\mu}}}_{b'}=0, \quad \mbox{at} \ x_{3}=0.
	\end{equation*}
	Therefore, we obtain 
	\begin{equation}\label{tildeemu=0}
	(\tilde{e}_{\mu})_{\tilde{j}b}=0, \quad \mbox{at} \ x_{3}=0
	\end{equation}
	which implies in particular that  $(\tilde{e}_{\mu})_{33}$ is the only remaining nonzero component at $x_3 = 0$. Considering \eqref{AQTEeq} while using the the symbols $A_{E}$, $Q_{E}$ and $T_{E}$, determined respectively from \eqref{Asym}, \eqref{Qsym} and \eqref{Tsym}, the hypothesis on the metrics as well as \eqref{tildeemu=0}, we can calculate
	\begin{equation*}
	\ast_{\hat{\mu}}\left (\nu \wedge \ast_{\hat{\mu}} ( \chi_{\hat{\mu}} \wedge \xi_{\hat{\mu}})\right ) \langle \nu, \xi_{\hat{\mu}} \rangle_{\hat{\mu}} (\tilde{e}_\mu)_{33} = 0.
	\end{equation*}
	Explicit calculation in coordinates using the definitions of $\chi_{\hat{\mu}}$ and $\xi_{\hat{\mu}}$ shows that the coefficient in the previous formula is not zero, and so we deduce 
	\begin{equation*}
	(\tilde{e}_{\mu})_{33}=0, \quad \mbox{at} \ x_{3}=0.
	\end{equation*}
	Therefore, we have shown that 
	\begin{equation*}
	\partial_{x_{3}}^{\kappa}\hat{\mu}^{-1}=\partial_{x_{3}}^{\kappa}\hat{\mu}'^{-1}, \quad \mbox{at} \ x_{3}=0,
	\end{equation*}
	which completes the proof.
\end{proof}

\begin{theorem}\label{orderofYJthm}
	Let $(\hat{\varepsilon},\hat{\mu})$ and $(\hat{\varepsilon}',\hat{\mu}')$ be two different sets of electromagnetic parameters and $\Lambda_{\hat{\varepsilon}}$, $\Lambda_{\hat{\varepsilon}'}$ the associated impedance maps. Assume that $B$, $B'$ are the factorisation operators corresponding to the magnetic fields $H$, $H'$ and let $Y=B'-B$. Assume further that $C$, $C'$ are the factorisation operators corresponding to the electric fields $E$, $E'$ and let $J=C'-C$. Let us fix boundary normal coordinates for $\hat{\varepsilon}$/$\hat{\varepsilon}'$. If
	\begin{equation*}
	\Lambda_{\hat{\varepsilon}}=\Lambda_{\hat{\varepsilon}'},
	\end{equation*}
	\begin{equation*}  
	\partial_{x^{3}}^{a}\hat{\varepsilon}= \partial_{x^{3}}^{a}\hat{\varepsilon}', \quad \partial_{x^{3}}^{a}\hat{\mu}=\partial_{x^{3}}^{a}\hat{\mu}', \quad \mbox{at} \ x_{3}=0,
	\end{equation*}
	for $a=0, \dots \kappa$
	and
	\begin{equation*}
	Y, J \in \Psi DO^{(1,\kappa)}(\partial M,\mathbb{R}^+),
	\end{equation*}
	then 
	\begin{equation*}
	Y, J\in \Psi DO^{(1,\kappa+1)}(\partial M,\mathbb{R}^+).
	\end{equation*}
\end{theorem}
\begin{proof}
The case $\kappa = 0$ follows from the formula \eqref{BJordan} and \eqref{CJordan} for $B$, $B'$, $C$ and $C'$ since the principal symbols depend only on the values of the respective parameters at the boundary. For $\kappa \geq 1$, it suffices to show that the equations
		\begin{equation}\label{SylvesterZ}
\left(	A_{H}+ T_{H}B^{(1)}\right)Z + T_{H}ZB^{(1)}= 0, 
	\end{equation}
	\begin{equation}\label{SylvesterW}
	\left(	A_{E}+ T_{E}C^{(1)}\right)W + T_{E}WC^{(1)}= 0, 
	\end{equation}
	admit the unique solutions $Z=0$ and $W=0$, respectively. This follows from the fact that after proving Theorems \ref{indHthm} and \ref{indEthm} the right hand-sides of the equations in \eqref{vectorkH} and \eqref{vectorkE} vanish. We will prove the result for \eqref{SylvesterZ} and the corresponding result for \eqref{SylvesterW} follows in a similar way.
	Equation \eqref{SylvesterZ} multiplied by $T_{H}^{-1}$ becomes
	\begin{equation}\label{Sylvestereq}
	\left( T_{H}^{-1}A_{H}+ B^{(1)}\right)Z + ZB^{(1)}= 0
	\end{equation}
	This is a homogeneous Sylvester equation with unknown $Z$. According to \cite[Theorem 8.2.1]{datta2004numerical}, if $T_{H}^{-1}A_{H}+B^{(1)}$ and $-B^{(1)}$ do not share any eigenvalues equation \eqref{Sylvestereq} admits the unique solution
	\begin{equation*}
	Z=0.
	\end{equation*}
We showed that $B^{(1)}$ satisfies
	\begin{equation*}
	T_{H}(B^{(1)})^{2}+ A_{H}B^{(1)}+ Q_{H}=0.
	\end{equation*}
	Hence we can write $T_{H}^{-1}A_{H}+B^{(1)}$ as
	\begin{equation*}
	T_{H}^{-1}A_{H}+B^{(1)}= -T_{H}^{-1} Q_{H}(B^{(1)})^{-1}.
	\end{equation*}
	So we have to investigate if $T_{H}^{-1}Q_{H}(B^{(1)})^{-1}$ and $B^{(1)}$ share any eigenvalues. Revisiting the factorisation of $M_{H}(\xi_{3})$ in terms of $B^{(1)}$ as given in equation \eqref{Mfactorization} we have
	\begin{equation*}
	T_{H}\xi_{3}^{2}+ A_{H}\xi_{3}+ Q_{H}= \left(\mathbb{I}\xi_{3}-(B^{\ast})^{(1)}\right)T_{H}\left(\mathbb{I}\xi_{3} -B^{(1)}\right).
	\end{equation*}
	Equating the terms that are of order zero with respect to $\xi_{3}$, we have that $Q_{H}$ is expressed as
	\begin{equation*}
	Q_{H}= (B^{\ast})^{(1)}T_{H} B^{(1)}.
	\end{equation*}
	Hence, $T_{H}^{-1}Q_{H}(B^{(1)})^{-1}$ is written as
	\begin{equation*}
	T_{H}^{-1}Q_{H}(B^{(1)})^{-1}= T_{H}^{-1}(B^{\ast})^{(1)}T_{H}.
	\end{equation*}
	Since the eigenvalues $\xi_{\hat{\mu}3}$ and $\xi_{\hat{\epsilon}3}$ of $B^{(1)}$ have nonzero complex component, the eigenvalues $\overline{\xi_{\hat{\mu}3}}$ and $\overline{\xi_{\hat{\epsilon}3}}$ of $T_{H}^{-1}(B^{\ast})^{(1)}T_{H}$ are different, which completes the proof.
\end{proof}

\section*{Acknowledgements}

We would like to thank Bill Lionheart for suggesting the problem and providing many useful discussions.

\bibliography{sn-bibliography}

\end{document}